\documentclass[reqno,12pt,a4paper]{article}
\pdfoutput=1
\usepackage[top=1in, bottom=1in, left=1in, right=1in]{geometry}
\usepackage{setspace}
\usepackage[font=small]{caption}
\usepackage[ngerman,british]{babel}
\usepackage{soul}
\usepackage[utf8]{inputenc}
\usepackage{authblk}
\usepackage[T1]{fontenc}
\usepackage{lmodern}
\usepackage{amsmath,amssymb,amsthm,amsfonts}
\usepackage{eurosym}
\usepackage{mathtools}
\usepackage{bbm}

\usepackage{nicefrac}       
\usepackage{xspace}
\usepackage{braket}
\usepackage{float}
\restylefloat{table}
\usepackage{booktabs}
\usepackage[short,nocomma]{optidef}
\usepackage{graphicx}
\usepackage{color}
\usepackage{epstopdf}
\DeclareGraphicsExtensions{.pdf,.eps,.png}
\usepackage{subfig}
\usepackage{stackengine}
\usepackage{todonotes}

\usepackage[bookmarksnumbered,bookmarksopen,unicode,colorlinks=true,allcolors = blue]{hyperref}
\usepackage{url}           

\usepackage[]{todonotes}

\newcommand{\N}{\mathbbm{N}}

\newcommand{\conv}[1]{\operatorname{conv}#1}
\newcommand{\ie}{i.e.\ }

\newcommand{\qm}[1]{``#1''}

\newcommand{\tb}[1]{\textcolor{black}{#1}}

\title{Optimized Noise Suppression for Quantum Circuits}

\author[1,3,*]{Friedrich Wagner}
\author[2]{Daniel J. Egger}
\author[1]{Frauke Liers}
\affil[1]{Department of Data Science, University of Erlangen-Nürnberg}
\affil[2]{IBM Quantum, IBM Research Europe – Zurich}
\affil[3]{Fraunhofer Institute for Integrated Circuits, Nürnberg}
\affil[*]{\texttt{\small friedrich.wagner@iis.fraunhofer.de}}

\newtheorem{lemma}{Lemma}

\begin{document}

	\maketitle
	% \tableofcontents
	\begin{abstract}
Quantum computation promises to advance a wide range of computational tasks.
However, current quantum hardware suffers from noise and is too small for error correction.
Thus, accurately utilizing noisy quantum computers strongly relies on noise characterization, mitigation, and suppression.
Crucially, these methods must also be efficient in terms of their classical and quantum overhead.
Here, we efficiently characterize and mitigate crosstalk noise, which is a severe error source in, e.g., cross-resonance based superconducting quantum processors.
For crosstalk characterization, we develop a simplified measurement experiment.
Furthermore, we analyze the problem of optimal experiment scheduling and solve it for common hardware architectures.
After characterization, we mitigate noise in quantum circuits by a noise-aware qubit routing algorithm.
Our integer programming algorithm extends previous work on optimized qubit routing by swap insertion.
We incorporate the measured crosstalk errors in addition to other, more easily accessible noise data in the objective function.
Furthermore, we strengthen the underlying integer linear model by proving a convex hull result about an associated class of polytopes, which has applications beyond this work.
We evaluate the proposed method by characterizing crosstalk noise for \tb{two chips with up to 127 qubits} and leverage the resulting data to improve the approximation ratio of the Quantum Approximate Optimization Algorithm by up to 10 \% compared to other established noise-aware routing methods.
Our work clearly demonstrates the gains of including noise data when mapping abstract quantum circuits to hardware native ones.
\end{abstract}

\section{Introduction}

Quantum computers may impact many disciplines such as natural sciences~\cite{Bauer_2020}, machine learning~\cite{Biamonte_2017,cerezo2022}, and optimization~\cite{Egger2021,Moll_2018, Abbas2023}. 
However, current quantum computing devices are noisy and their qubit count is too low for quantum error correction which requires a large overhead in resources~\cite{nielsen_chuang_2010}.
By contrast, quantum error mitigation (QEM) executes an ensemble of noisy circuits and performs a classical post-processing to deliver a noise mitigated result of, typically, an expectation value~\cite{Sack_2023,van_den_Berg_2022,Temme_2017}.
Similarly, randomized compiling simplifies the noise structure such that classical post-processing can be applied to reduce noise in derived measurement quantities \cite{Kern_2005,Cai_2019,Wallman_2016,Hashim_2021, Kim2023}.
QEM thus requires an overhead in classical resources and quantum samples.
Quantum error suppression modifies hardware instructions and is less resource demanding.
For example, dynamical decoupling (DD) is a well-established error suppression method that inserts carefully chosen sequences of single-qubit gates, which evaluate to the identity~\cite{Tripathi_2022,Viola_1998}.
Noise-aware qubit-routing suppresses errors in the compilation process~\cite{Murali_2019,Nishio_2020,Niu_2020,Sarovar2020,Hua_2022}.
Error mitigation and suppression therefore allow noisy quantum computers to deliver meaningful results at scale~\cite{Kim2023}. 
Crucially, noise-aware compilation relies on a precise and efficient preceding noise characterization.

\paragraph{Existing noise characterization \tb{methods}.}
While quantum applications run on all or a large fraction of the qubits in quantum processors, gates and qubits are often benchmarked in isolation.
For example, in superconducting quantum computers, basic error data is measured by standardized daily calibration routines~\cite{IBMQuantum}.
This includes average single- and two-qubit error rates, readout error rates and qubit coherence times.
Ramsey experiments characterize qubit frequency and coherence times~\cite{Krantz_2019}.
Randomized benchmarking (RB) is an established protocol to characterize average \tb{single- and two-qubit} gate error rates~\cite{Magesan_2011,Magesan_2012,Magesan_2012_2,McKay_2019}.
\tb{
However, it is surprisingly hard to scale RB to a large number of qubits~\cite{Proctor_2019}.
By contrast, direct RB can characterize the average error rates of a processor's native gates on more than a few qubits~\cite{polloreno2023theory,Proctor_2019}.}
\tb{Furthermore, cycle benchmarking extends RB to efficiently quantify average error rates of multi-qubit operations~\cite{Erhard_2019,carignandugas2023error}.}
\tb{However, }these metrics \tb{may not provide enough details on the}
\emph{crosstalk} \tb{in an application quantum circuit }which is a severe error source and requires more elaborate experiments to characterize~\cite{Dai_2021,Abrams_2019,Murali_2020,Guan_2023,Rudinger_2021}.
In the literature, the term crosstalk is used ambiguously for a variety of noise phenomena originating from unwanted interactions in quantum information processors~\cite{Sarovar2020}.
For example, superconducting qubit devices based on fixed-frequency transmon qubits suffer from a static interaction between qubits of the form $e^{-i\theta Z\otimes Z}$, where $\theta$ denotes a rotation angle and $Z$ the Pauli Z-matrix. This is often referred to as ZZ crosstalk~\cite{Ash_Saki_2020,Tripathi_2022,Xie_2022,Wei_2022}.
By contrast, dynamic crosstalk is triggered by gate execution.
Here, frequency collisions of computational or non-computational state transitions in qubits spatially close to the driven qubits lead to an unwanted dynamics~\cite{Ketterer_2023, Heya2023, Ding_2020}.
As a result, the error rates for two-qubit gates executed in parallel increase compared to when they are executed independently.
This is sometimes referred to as CX-CX crosstalk with CX referring to the controlled-NOT gate~\cite{Murali_2020,Ash_Saki_2020,Hua_2022,Guan_2023}.
Similarly, single-qubit error rates may also increase when neighboring two-qubit gates are applied simultaneously~\cite{Ketterer_2023,Perrin_2023}.
To distinguish this effect from CX-CX crosstalk, we use the term CX-SQ crosstalk with SQ referring to single-qubit.
%Thus, CX-SQ crosstalk is a perturbation induced by a two-qubit gate on another single-qubit.
Increased error rates caused by CX-CX crosstalk are quantifiable via simultaneous randomized benchmarking (SRB)~\cite{Gambetta_2012, Murali_2020,Hua_2022,Guan_2023,Rudinger_2019, Rudinger_2021}.
SRB performs RB in parallel on the two-qubit gate pair of interest.
Analogously, SRB can also determine increased error rates due to CX-SQ crosstalk~\cite{Ketterer_2023}.
However, a full characterization of CX-SQ or CX-CX crosstalk for a given device
requires a careful planning of SRB experiments to keep the number of required circuits tractable~\cite{Murali_2020}.

\paragraph{Existing noise-aware \tb{transpilation methods}.}
Once noise and crosstalk are characterized, faulty hardware components can be avoided by noise-aware qubit routing, a step in the \emph{transpilation} process.
Transpilation subsumes all processes which transform a quantum circuit into a logically equivalent one, typically performing optimization steps.
Herein, qubit routing is the task of transforming
a circuit into an equivalent one which meets any hardware connectivity restrictions.
Often, this is achieved by first defining an initial mapping of circuit qubits to hardware qubits.
Next, swap gates are inserted such that qubits involved in two-qubit gates are physically adjacent at some point in the circuit~\cite{Sivarajah_2020,Li2019,zulehner2019}.
Noise data can be incorporated in both the initial mapping and the swap insertion.
Murali et al.~\cite{Murali_2019} propose a heuristic to determine a hardware-subgraph with low noise levels for the initial mapping.
The Tket compiler also offers a heuristic to choose a low-noise subgraph~\cite{Sivarajah_2020}.
Niu et al.~\cite{Niu_2020} propose a routing algorithm where only the swap insertion procedure is noise-aware.
Nishio et al.~\cite{Nishio_2020} propose a routing method which considers noise in both the initial choice of a subgraph and the consecutive swap insertion procedure.
Notably, all noise-aware routing methods mentioned so far do not consider crosstalk.
Hua et al.~\cite{Hua_2022}, on the other hand, propose a CX-CX crosstalk aware routing method which considers crosstalk only in the swap insertion phase.
\tb{Booth et al.~\cite{Booth2018} propose a routing method which considers different crosstalk types in the initial layout and swap insertion.}
\tb{However, their method is insensitive to varying crosstalk strength among different qubits.}
\tb{Khadirsharbiyani et al.~\cite{Khadirsharbiyani_2023} develop an initial layout method to reduce CX-CX crosstalk.}
\tb{Xie et al.~\cite{Xie_2021} re-order gates based on commutativity rules to reduce CX-CX crosstalk.}
\tb{Importantly, all crosstalk-aware methods mentioned so far do not consider other error data like two-qubit errors, coherence times or readout errors.}
Gate scheduling, i.e., defining the exact execution times of gates without changing their order, can also help mitigate noise~\cite{smith2021error}.
In particular, gate-triggered crosstalk can sometimes be avoided by delaying gates~\cite{Murali_2020, Guan_2023, niu2021analyzing} which, however, comes at the cost of an increased execution time which in turn leads to larger decoherence noise.
\tb{On the other hand, Xie et al.~\cite{Xie_2022} use gate scheduling to reduce static ZZ-crosstalk.}
Tripathi et al.~\cite{Tripathi_2022} \tb{and Zhou et al.~\cite{Zhou_2023}} reduce static ZZ-crosstalk with DD, but do not consider dynamic crosstalk suppression.
\tb{Fang et al.~\cite{Fang_2022} partially undo gate-triggered crosstalk by inserting single-qubit gates on a trapped-ion processor.}
Because of its high impact, crosstalk is even considered in quantum hardware design and modeling~\cite{Sarovar2020, Zhao_2022, huang2020alibaba, Winick_2021}.

\paragraph{Our contribution.}
This work contributes to efficient crosstalk characterization and measurement.
Additionally, it exploits the obtained findings for high-quality noise suppression.
First, we simplify SRB experiments to quantify CX-SQ crosstalk.
Instead of running random single- and two-qubit gate sequences in parallel, we replace the random two-qubit gate sequence by a single, appropriately stretched cross-resonance pulse~\cite{Rigetti2010}.
This avoids the time consuming compilation of long sequences of random two-qubit gates.
Moreover, we propose an optimal experiment scheduling protocol for measuring crosstalk of a full device.
To this end, we formulate the task of determining the number of necessary experiments as a graph coloring problem.
Although graph coloring is an NP-hard problem in general, we show that for a common hardware architecture family, it can be solved analytically.
For other common families, experiments show that integer programming solves the coloring problem to optimality within reasonable time and only requires a constant number of colors, independently of the hardware size.
As a result, only a constant number of circuits is required 
for full CX-SQ crosstalk characterization on common architectures of arbitrary size.
Concerning noise suppression, we enhance and extend a routing method based on integer programming that considers both standard calibration data and crosstalk data.
\tb{This is in contrast to existing noise-aware routing approaches which focus on either crosstalk or standard calibration data.}
\tb{Contrary to QEM techniques, our method reduces noise in individual samples rather than expectation values.}
We build upon the integer programming approach of Ref.~\cite{WBLW_2023}.
We first strengthen the underlying linear model by providing an outer description of the convex hull for a closely related class of polytopes, which generalizes to other applications in operations research.
\tb{As a result, the integer programming runtime is reduced such that practically relevant instances can be solved in reasonable time although the underlying problem is NP-hard.}
Moreover, we extend the objective function with additional binary quadratic terms to incorporate noise data.
We evaluate \tb{the} proposed methods on the 27 qubit \emph{ibmq\_ehningen} device \tb{and on the 127 qubit \emph{ibm\_kyoto} device}.
First, the applicability of the characterization procedure is shown by characterizing crosstalk for the complete devices.
Additionally, we evaluate the combination of noise characterization and suppression by improving the performance of the Quantum Approximate Optimization Algorithm (QAOA,~\cite{farhi_2014quantum}) \tb{on \emph{ibmq\_ehningen}}.
The experiments show that the proposed approach suppresses noise more effectively than exiting methods for noise-aware routing.

\paragraph{Structure.}
The reminder of this paper is structured as follows.
We develop the new method to characterize crosstalk and optimally schedule experiments in Section~\ref{sec:xt_char}.
Building upon this, we enhance and extend the existing routing algorithm to incorporate crosstalk errors and other noise data in Section~\ref{sec:mitigation}.
In Section~\ref{sec:exp}, we perform computational experiments that use the noise values obtained earlier.
We evaluate the developed characterization and suppression tools on real quantum hardware and show that the obtained results are improved compared to the existing noise-aware routing methods.
We end with a conclusion in Section~\ref{sec:concl}.

\section{Optimized Crosstalk Characterization}\label{sec:xt_char}
A precise and efficient quantification of crosstalk is crucial to mitigate it.
Moreover, regular characterization routines are necessary since noise characteristics may fluctuate over time~\cite{Murali_2020}.
Thus, to minimize the effort of characterizing crosstalk one should use a small number of simple experiments.
To this end, we propose a simplified version of SRB in Section~\ref{sec:srb} which avoids long sequences of random two-qubit gates.
We perform RB on single-qubit gates while a single long-duration two-qubit gate is applied on neighboring qubits.
Additionally, in Section~\ref{sec:coloring} we derive an optimized experiment schedule for crosstalk characterization of a complete device.
We model the scheduling problem as a graph coloring problem on an interference graph.
For graphs arising from heavy-hexagonal architectures, a common device family, we construct its optimum solution analytically.
Here, a constant number of colors suffices, even for infinite graphs.
Thus, the resulting scheduling protocol is optimal in terms of executed circuits,
which is constant, independent of the hardware size.
For other common hardware architectures consisting two-dimensional grids and six-regular graphs, we solve the coloring problem by integer programming.
\tb{Finally, in Section~\ref{sec:xt_char_res}, we demonstrate the applicability of the developed characterization and scheduling methods
by fully characterizing crosstalk noise for the chips of \emph{ibmq\_ehningen} and \emph{ibm\_kyoto}.}

\subsection{Crosstalk Characterization by Simultaneous Randomized Benchmarking\label{sec:srb}}
In superconducting qubit devices gates are executed by applying microwave pulses.
Here, two-qubit gates may trigger crosstalk due to frequency collisions, i.e., similar transition frequencies of neighboring qubits~\cite{Ketterer_2023,Murali_2020,Guan_2023}.
These chips are thus carefully designed.
For instance, the frequency allocation problem can be formulated as a mixed-integer programming problem~\cite{Morvan_2022}.
Here, we work with hardware based on the cross-resonance (CR) interaction~\cite{Rigetti2010}, from which a CX gate is built.
When applying a CR gate, one of the two involved qubits is driven at the frequency of the other.
As a result, frequency collisions between transitions of next-nearest neighbors are relevant and may lead to an unwanted driving of spectator qubits.

Standard RB determines the average error per $n$-qubit gate, where $n\leq 2$ throughout this work.
It can be extended to measure crosstalk errors between two disjoint sets of qubits.
To this end, one first performs RB on both sets sequentially.
Afterwards both experiments are repeated, but performed simultaneously.
The observed increase in the average error per gate is a direct measure of the crosstalk strength~\cite{Murali_2020, Guan_2023,Rudinger_2021, Gambetta_2012}.
For example, CX-SQ crosstalk between the CX gate on qubits $(i,j)$ and qubit $k$ is characterized by first performing single-qubit RB on qubit $k$ to measure the average error per single-qubit gate on qubit $k$.
Next, two-qubit RB on qubits $(i,j)$ yields the average error per gate on qubits $(i,j)$.
Finally, both experiments are performed simultaneously and any increase in average error per gate is attributed to crosstalk~\cite{Ketterer_2023}.
We refer to this method of characterizing crosstalk as SRB.
SRB requires compiling and executing many long sequences of two-qubit gates, which \tb{can be} time consuming~\cite{KHANEJA2001}.
\tb{Here, direct RB can reduce the compilation overhead~\cite{Proctor_2019,polloreno2023theory}.
However, our method entirely avoids the compilation of two-qubit gates and generalizes to multi-qubit gates.}

We simplify the measurement of crosstalk via SRB in two ways.
First, it is sufficient to measure the influence of each CX gate on its neighboring qubits.
Indeed, the reverse characterization, i.e., the influence of single-qubit gates on neighboring CX gates (termed SQ-CX crosstalk), is not necessary since our experiments, presented in Appendix~\ref{sec:app_xt}, show that this effect is an order of magnitude smaller than CX-SQ crosstalk.
Analogous experiments in Appendix~\ref{sec:app_xt} show that the same holds true for crosstalk among single-qubit gates (SQ-SQ crosstalk).
Moreover, similar experiments, also shown in Appendix~\ref{sec:app_xt}, reveal that CX-CX crosstalk can be attributed to CX-SQ crosstalk in large parts and is thus captured implicitly by quantifying CX-SQ crosstalk.
In summary, we do not need to determine two-qubit error rates at all, which drastically reduces the number of experiments.
Thus, in the example above, we skip the second step which determines the error rate for gate $(i,j)$.
Moreover, since the two-qubit gate error rate is not needed we replace the two-qubit gate sequence by a single, appropriately stretched cross-resonance pulse~\cite{Earnest2021}, schematically shown in Fig.~\ref{fig:schematic}.
Summarizing, for CX-SQ crosstalk between CX $(i,j)$ and qubit $k$, we perform only two experiments:
standard RB on qubit $k$ and RB on qubit $k$ with a simultaneous stretched CR pulse applied to qubits $i$ and $j$, see Fig.~\ref{fig:srb_ex} for an example.
We refer to this simplified RB protocol as CXRB.

\begin{figure}
    \captionsetup[subfigure]{position=top, labelfont=bf,textfont=normalfont,singlelinecheck=off,justification=raggedright, skip=0pt,margin = {00pt, 20pt}}
    \makebox[\textwidth][c]{
    \subfloat[]{\includegraphics[width=0.33\textwidth]{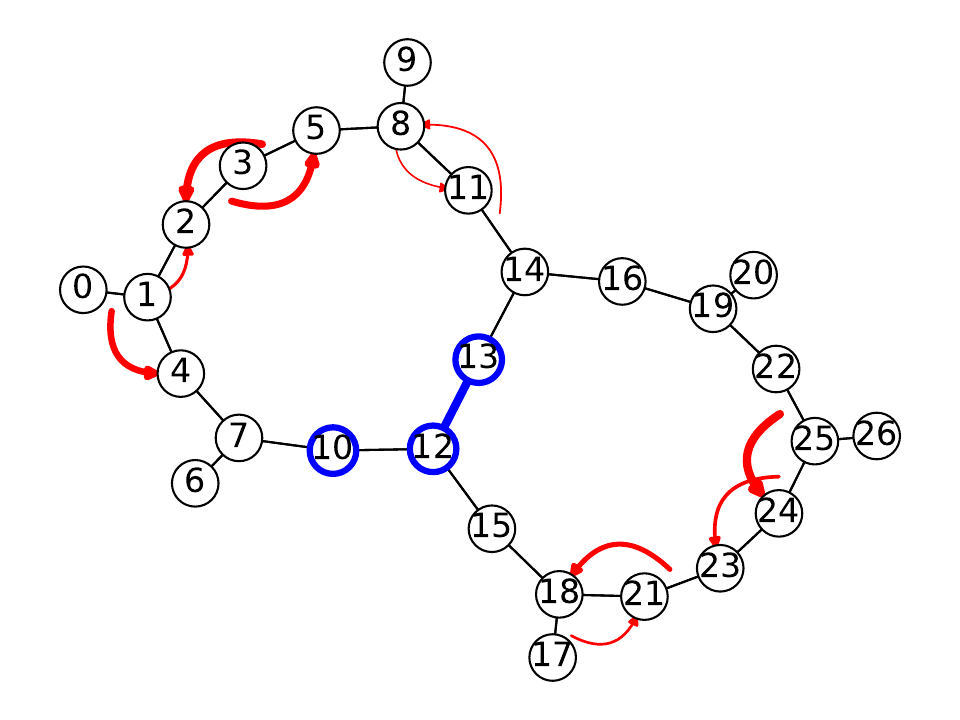}\label{fig:cm_ehningen}}
 	\subfloat[]{\includegraphics[width=0.33\textwidth]{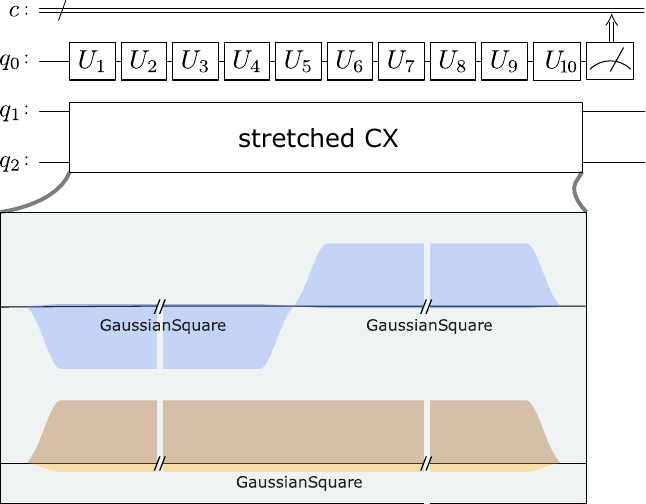}\label{fig:schematic}}
    \subfloat[]{\includegraphics[width=0.33\textwidth]{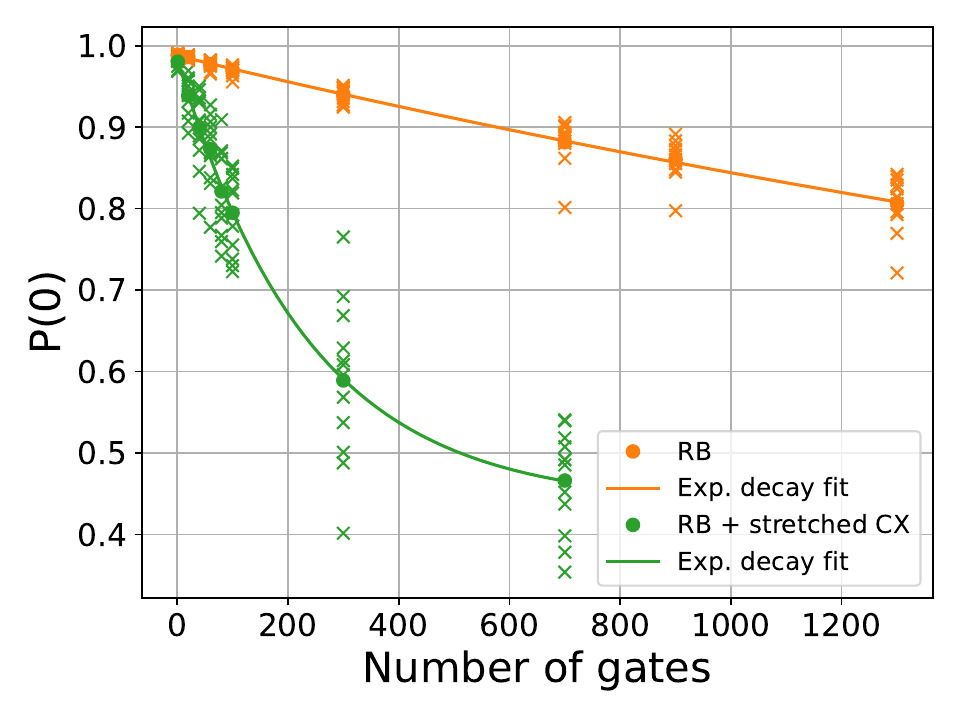}\label{fig:srb_ex}}} \\
    \subfloat[]{\includegraphics[width=1\textwidth]{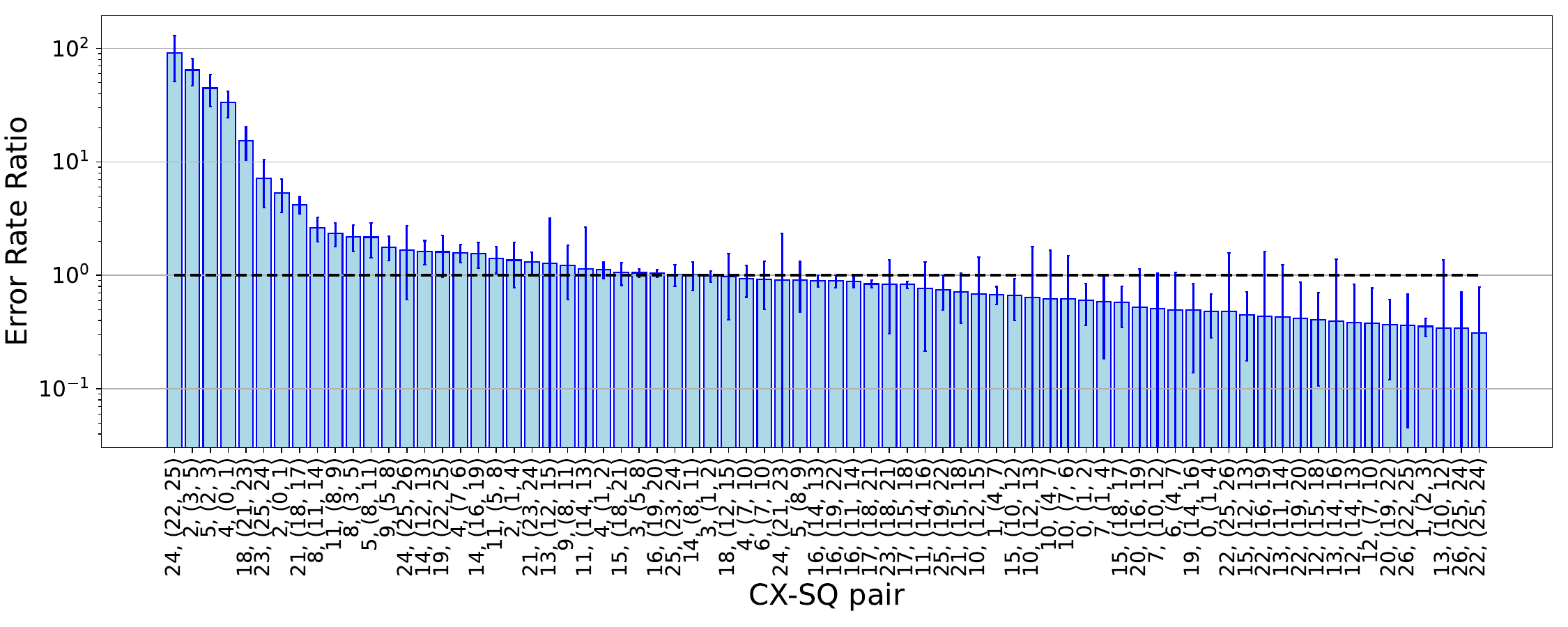}\label{fig:hist_cxrb}}
    
	\caption{ (a)~Coupling graph of \emph{ibmq\_ehningen}, a subgraph of a heavy hexagonal lattice. Red arrows indicate large CX-SQ crosstalk, where the thickness is proportional to the crosstalk magnitude. In bold blue, a CX-SQ pair (connected vertex triplet) is marked. In total, there exist 74 such pairs, which need to be characterized for crosstalk.
 (b)~Quantum circuit of the CXRB experiment.
 A standard RB sequence of random gates is applied on $q_0$.
 In parallel, a stretched CR pulse is applied to qubits $q_1$ and $q_2$ to mimic the effect of ${\rm CX}_{1, 2}$.
 (c)~Exemplary results of a CXRB experiment. Crosses mark single measurements, dots are averages and solid lines are exponential decay fits.
 The decay rate of a fit curve directly relates to the average error rate per gate, \tb{see Appendix \ref{sec:app_rb} for details}.
 We apply less random gates for RB with parallel stretched CX (green, lower curve) since we expect a steeper decay compared to standard RB (orange, upper curve).
 (d)~CX-SQ magnitude for the complete chip of \emph{ibmq\_ehningen}. For each CX-SQ pair, e.g. the blue nodes in (a), we give the ratio between the error rate with and without applied CR pulse (ERR).
 Error bars \tb{are obtained by performing an error propagation from the errors in the exponential decay fits}. 
 The dashed line corresponds to $\text{ERR}=1$, i.e., no crosstalk.
 }
	\label{fig:srb_schematic}    
\end{figure}

\subsection{Optimal Experiment Scheduling via Graph Coloring}\label{sec:coloring}
Having developed a simplified CX-SQ crosstalk measurement technique, we now derive a procedure to characterize a complete chip with a minimal number of circuits.
Characterizing a complete chip amounts to measuring the influence of every native CX gate on all of its neighboring qubits via CXRB experiments, see Fig.~\ref{fig:cm_ehningen}.
To execute as few circuits as possible we parallelize experiments.
First, we note that the influence of a given CX on all neighboring qubits can be
characterized in parallel.
To this end, we perform RB circuits on all neighboring single qubits in parallel.
Next, we repeat RB but with a stretched CR pulse on the appropriate qubits.
Moreover, two CX gates can be characterized in parallel if they do not interfere, that is, if they do not share a common neighbor.
A common approach for such planning tasks is to model the problem as a graph coloring problem on an interference graph~\cite{Marx_2004}, which is an NP-hard problem in general.
In a graph coloring, each vertex is assigned a color such that its edges only connect vertices having different colors.
The minimal number of colors required to color a given graph is called its chromatic number.
In our case, the interference graph has a vertex for every edge in the hardware graph, representing the CX gates to characterize.
Two vertices are connected if the corresponding CX gates interfere as defined above. 
An optimal vertex coloring of the interference graph, i.e. a coloring with smallest number of colors, now corresponds to an experiment schedule with a minimal number of circuits.
The authors of Ref.~\cite{Murali_2020} employ a randomized greedy coloring algorithm to tackle the coloring problem heuristically.
The greedy algorithm initializes all vertices uncolored.
Then, the vertices are traversed in a random order and each vertex is assigned the smallest feasible color.
The procedure is repeated multiple times and the best coloring is returned.

Instead of a coloring heuristic, we use an exact coloring algorithm which bears several advantages.
First, a solution with less colors directly reduces the overhead needed to characterize crosstalk.
Second, the optimal coloring is only computed once per device.
Moreover, most current quantum devices can be grouped into architecture \emph{families}.
All device architectures in a family are subgraphs of the same infinite graph.
Examples include two dimensional grids~\cite{Arute_2019} and heavy-hexagonal lattices~\cite{Chamberland_2020}.
The authors of Ref.~\cite{bravyi2023} propose novel architectures based on six-regular graphs.
It is not hard to see that the interference graphs of all architectures in a family are subgraphs of the interference graph of the infinite graph.
Furthermore, a coloring of the infinite graph naturally induces a coloring for every subgraph.
The induced coloring will not be optimal in general.
However, we show that this is indeed the case for heavy-hexagonal lattices if the subgraph exceeds a certain minimum size.
For grids, we resort to a suitably large finite graph, containing all existing architectures as subgraphs, and also show that the induced colorings are optimal
if the subgraphs exceeds a certain size.
Finally, since the coupling of current quantum devices is sparse, also the corresponding interference graphs are sparse,
such that integer programming methods can solve the coloring problem to optimality for all practically relevant instances in reasonable time.

As a relevant example, the architecture of the device used in this work is a subgraph of a heavy-hexagonal lattice.
A heavy-hexagonal lattice is a hexagonal lattice where an additional vertex is inserted in every edge as shown in Fig.~\ref{fig:heavy_hex_coloring}.
We refer to a single hexagon, consisting of 12 vertices, as a unit cell.
With this notion, the following Lemma applies.
\begin{lemma}\label{lem:hex_color}
    Let $G$ be a finite subgraph of the infinite heavy-hexagonal graph
    and let $L(G)$ be its interference graph.
    If $G$ contains two connected unit cells as a subgraph, then $L(G)$ has chromatic number $\chi(L(G))=6$.
\end{lemma}
\begin{proof}
    First, we note that if $G$ contains two connected unit cells as a subgraph,
    $L(G)$ has a clique of size six, see Fig.~\ref{fig:heavy_hex_coloring}.
    Thus, $\chi(L(G))\geq 6$.
    Let $\tilde{G}$ be the infinite heavy-hexagonal graph.
    We construct a proper six-coloring of every finite subgraph of the interference graph $L(\tilde{G})$ in the following way.
    Let the six colors be numbered from 1 to 6 and
    choose an arbitrary order of the hexagons in $\tilde{G}$.
	In every hexagon, we enumerate its first six edges, starting in the lower-left and proceeding clock-wise and color it with the respective color.
	For every edge, there is exactly one hexagon such that the edge is among the hexagon's first six edges.
    Thus, every edge holds exactly one color.
    It is easily verified that this yields indeed a proper six-coloring of $L(\tilde{G})$, see Fig.~\ref{fig:heavy_hex_coloring}.
\end{proof}
% In Fig.~\ref{fig:heavy_hex_coloring}, we visualize an optimal coloring.
As a consequence of Lemma~\ref{lem:hex_color}, every hardware chip with heavy-hexagonal architecture can be characterized for CX-SQ crosstalk
by performing six consecutive CXRB experiments.
Moreover, if the chip contains two hexagonal unit cells, this is the minimal number of experiments.
By contrast, the best solution we found in $10,000$ runs of the randomized greedy algorithm uses 9 colors instead of the best possible number of 6 for a heavy-hexagonal lattice with 25 unit cells arranged in a $5 \times 5$ grid which has 164 vertices.
The latter is approximately the size of the largest existing heavy-hexagonal device~\cite{IBMQuantum}.

For 2D-grids, we resort to a suitably large finite lattice of size $11\times 11$ which contains the largest currently existing 2D-grid quantum architectures as a subgraph~\cite{Acharya_2023}.
Next, we model the graph coloring instance as an integer linear program, see e.g. Ref.~\cite{Lewis2016},
which is solved via an available state-of-the-art solver for mixed-integer programming~\cite{gurobi} within roughly 80 seconds. % No sym break, no textbook (-> big M), fix max clique
The minimal number of colors is 16, almost three times as large as for heavy-hexagonal lattices.
Moreover, the induced coloring is optimal for every subgraph containing a $5\times 5$ grid, since they contain a clique of size 16, see Fig.~\ref{fig:grid_coloring}.
Here, the best solution found in $10,000$ runs of the greedy algorithm employs more than the necessary 16 colors, namely 22 colors.

The simplest coupling map example in Ref.~\cite{bravyi2023} is a degree-six regular graph with 144 vertices and 432 edges.
Here, solving the graph coloring integer linear program took roughly 24~hours % No sym break, no textbook (-> big M), fix max clique; 6 hours with heurisitcs = mipfocus = 1
which is still an acceptable runtime since the integer program needs to be solved only once per device or device family, as discussed above.
\tb{Moreover, tuning the solver parameters to focus more on finding good solutions rather than generating bounds reduced the runtime to six hours.}
\tb{For practical applications, we can even interrupt the solver early and take the best solution found so far.
    Often, this yields a close-to-optimal solution.}
% \tb{For example, a solution to the degree-six regular graph with 51 colors was found after 0~s.}
\tb{Finally, theoretical analysis of the underlying problem can improve the IP solution time significantly as we show in Sec.~\ref{sec:polyhedral}.}
The IP solution reveals that an optimal coloring uses 36 colors.
By contrast, the best solution found in $10,000$ runs of the greedy algorithm employs 62 colors, i.e., almost twice as many.

In summary, we can find the exact minimum number of experiments to characterize crosstalk without resorting to heuristics by exploiting the regular nature of the coupling map of a quantum device.
Indeed, our experiments reveal that the randomized greedy heuristic typically finds a coloring that needs considerably more numbers than the optimum solution on all three tested architectures.
Furthermore, we observe a strong increase of the overhead required for crosstalk characterization of denser architectures which typically require more colors than sparse architectures.

\begin{figure}
	\centering
	\subfloat[]{\includegraphics[width=0.4\linewidth]{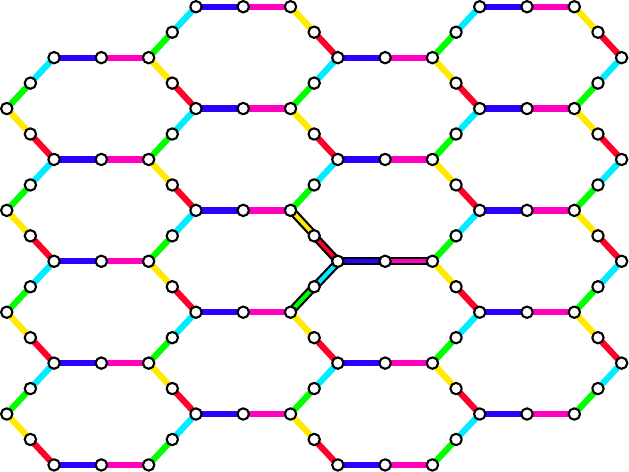}\label{fig:heavy_hex_coloring}}\qquad
    \subfloat[]{\includegraphics[width=0.4\linewidth]{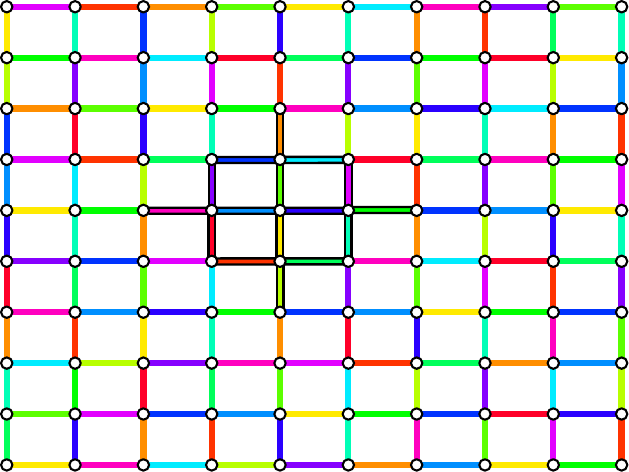}\label{fig:grid_coloring}}
	\caption{Visualization of the efficient crosstalk characterization protocol for a heavy-hexagon (a) and a grid architecture (b).
    Identically colored edges are characterized in parallel.
    The heavy-hexagonal structure requires only six experiments, whereas the grid needs 16. Edges marked with solid black lines form a clique of size six and 16 in the corresponding interference graphs.}
	\label{fig:coloring}
\end{figure}

\subsection{Hardware Characterization}\label{sec:xt_char_res}
We demonstrate the applicability of CXRB by characterizing CX-SQ crosstalk for the complete \emph{ibmq\_ehningen} chip.
The hardware graph, shown in Fig.~\ref{fig:cm_ehningen}, has 27 qubits connected by 28 resonators.
A total of 74 CX-SQ pairs, i.e., connected edge-vertex pairs in Fig.~\ref{fig:cm_ehningen}, exist whose crosstalk we characterize.
Using the optimal experiment schedule derived in Section~\ref{sec:coloring}, this is achieved by executing only six consecutive batches of CXRB.
The experiments are implemented with the open source framework Qiskit Experiments~\cite{Kanazawa2023}.
From the resulting data, we compute the average error per single-qubit gate with and without the parallel CX drive and compute their ratio.
We refer to this ratio as the Error Rate Ratio (ERR).
For most CX-SQ pairs the ERR is close to one, i.e., there is no significant crosstalk, see Fig.~\ref{fig:hist_cxrb}.
The ERR is larger than one for only 13 pairs at a statistical significance level of 95~\%.
Moreover, we observe an $\text{ERR}>10$ for five pairs which are thus severely impacted by crosstalk.
Values of $\text{ERR}<1$ are likely due to measurement uncertainties.

To show that the simplified CXRB protocol can indeed replace SRB, we additionally perform SRB experiments for CX-SQ characterization on the complete chip.
The crosstalk measured with CXRB and SRB have a correlation coefficient of $0.95$, see Fig.~\ref{fig:cxrb_srb_cor}.
\tb{To quantify the statistical significance of the inferred correlation coefficient, we perform a statistical test for
the null-hypothesis that uncorrelated, normally distributed data would yield a correlation coefficient at least as large.
The test yields a a $p$-value of $1.0\cdot 10^{-38}$ which shows that the correlation is highly significant.}
We therefore conclude that if SRB detects crosstalk then so does CXRB.

Finally, to demonstrate scalability and generalizability of the simplified CXRB protocol,
we additionally characterize CX-SQ crosstalk for the complete 127-qubit chip of {\emph{ibm\_kyoto}}.
This device has a total of 394 CX-SQ pairs whose crosstalk we characterize.
Even with more than five times as many CX-SQ pairs as \emph{ibmq\_ehningen},
we only need to execute six consecutive batches of CXRB to fully characterize CX-SQ crosstalk on the complete chip.
This is achieved via the optimal experiment schedule derived in Section~\ref{sec:coloring}.
Similar to \emph{ibmq\_ehningen}, the ERR is close to one for most CX-SQ pairs, see the data in Appendix~\ref{sec:app_kyoto}.
We observe an ERR larger than one for only {30} out of the 394 pairs at a statistical significance level of 95~\%. % osaka 43
Additionally, we perform CX-SQ characterization via standard SRB for the complete chip of {\emph{ibm\_kyoto}}.
The correlation coefficient between the ERRs measured with CXRB and SRB computes to $0.32$, which is smaller than for \emph{ibmq\_ehningen}
but still highly significant with a $p$-value of $3.0\cdot 10^{-6}$, see Appendix~\ref{sec:app_kyoto}. % osaka statistic=0.21547039373200327, pvalue=2.0054745054771654e-05
We thus conclude that the simplified CXRB protocol is scalable to larger devices and reliably detects crosstalk.

\begin{figure}
	\centering
    \includegraphics[width=0.5\linewidth]{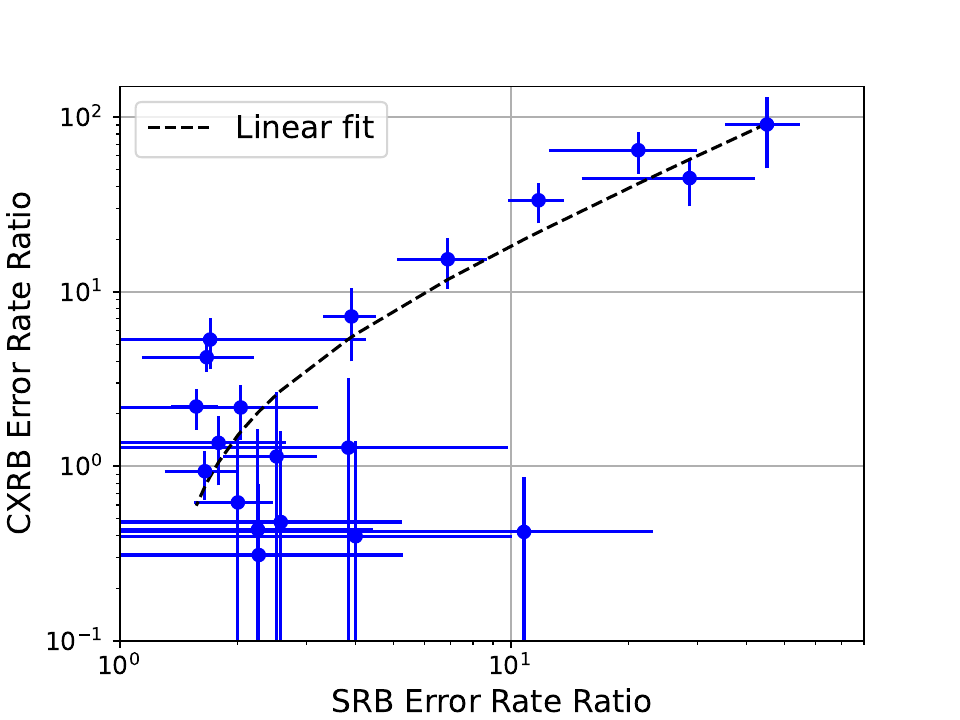}
	\caption{
    Correlation between the error rate ratio measured via SRB and CXRB. 
    For clarity, only points corresponding to the $20$ largest SRB ERRs are depicted.
    Error bars mark one standard deviation.
    The Pearson correlation coefficient computes to $0.95$ (all data points included), revealing a large linear correlation.
    }
	\label{fig:cxrb_srb_cor}
\end{figure}

\section{Crosstalk Mitigation via Integer Programming}\label{sec:mitigation}
We now develop a noise-aware qubit routing algorithm which, besides single-qubit, two-qubit and readout errors, also accounts for crosstalk errors.
In general, qubit routing methods take a quantum circuit and a hardware connectivity graph as input and return a hardware-compliant quantum circuit.
This circuit is logically equivalent to the input circuit up to basis permutations, i.e.~relabeling qubits.
First, one defines an initial mapping from circuit qubits to hardware qubits.
Finding a good initial mapping, e.g., to minimize swap overhead is an NP-hard problem~\cite{MATSUO_2023,Peham2023}.
Next, if needed, swap gates are inserted in the circuit, effectively changing the mapping from circuit qubits to hardware qubits, such that circuit qubits involved in two-qubit gates are always mapped on connected hardware qubits.
For example, heuristics, such as SABRE~\cite{Li2019}, iteratively refine the initial mapping and the swap gate insertion.
Naive objectives for routing are either swap count or circuit depth.
By contrast, noise-aware methods use more complex quality metrics that incorporate hardware noise data to estimate the performance of the routed circuit on the target hardware~\cite{Murali_2019,Nishio_2020,Niu_2020,Sarovar2020,Ferrari_2022}.
The measured crosstalk in Section~\ref{sec:srb} motivates a noise-aware routing method to reduce errors.
Intuitively, crosstalk is avoidable via routing at a low swap cost if the subset of gate-qubit pairs suffering from large crosstalk is small.
This is the case for \emph{ibmq\_ehningen} where only 5 pairs have a large crosstalk, see Fig.~\ref{fig:srb_schematic}.

\subsection{Qubit Routing via Integer Programming}\label{sec:polyhedral}
We build on the routing algorithm {TAP+TS} proposed in Ref.~\cite{WBLW_2023}.
Here, TAP+TS refers to token allocation problem and token swapping problem, two NP-hard optimization problems whose solutions are the central building blocks of the algorithm.
TAP+TS is itself based on the exact binary linear programming approach of Ref.~\cite{Nannicini_2023}.
Typically, such methods need a considerably smaller number of swap gates than state-of-the art heuristics at the expense of an increased running time.
However, TAP+TS is still faster by a factor of 100 on average than exact methods, which are intractable even for moderate input sizes with less than ten qubits~\cite{Nannicini_2023,Wille_2019}.
Furthermore, in contrast to other heuristics, it returns provable bounds on the quality of the obtained solution.
We now summarize the three phases of the TAP+TS algorithm.
\begin{enumerate}
    \item All two-qubit gates are grouped into layers.
        A layer is a set of gates on disjoint qubits that can be executed in parallel.

    \item After grouping, the token allocation problem (TAP) is solved
        via binary linear programming.
        This is the computationally most expensive step in the algorithm.
        The solution of the TAP is a mapping (allocation) from circuit qubits to hardware qubits for each layer.
        Here, the objective function to minimize is a lower bound on the total number of swaps required. 

    \item Finally, after allocating qubits in each layer, swaps between consecutive layers are inserted to transform between consecutive qubit allocations.
        For each pair of consecutive allocations, this task forms an instance of the token swapping problem, which is also NP-hard in itself.
        The token swapping problems are solved by an efficient approximation algorithm, an improved version of the algorithm originally proposed in Ref.~\cite{Miltzow_2016}.
\end{enumerate}

Here, we enhance this routing algorithm in two ways.
First, we improve the TAP binary linear model compared to Ref.~\cite{WBLW_2023} by giving a complete description of the convex hull of specific constrained binary quadric polytopes.
On the theoretical side, this result generalizes beyond our application.
Computationally, for our instances it improves the running time when solving the model with a branch-and-cut algorithm by about a factor of two on average.
Second, we incorporate noise data in the cost function of the model.
Besides single-qubit, two-qubit and readout error rates, we also incorporate crosstalk.

To ensure our work is self-contained, we summarize the TAP binary linear programming model from Ref.~\cite{WBLW_2023}. 
Afterwards, we strengthen the model and extend the objective function to incorporate noise data via additional quadratic terms.
We consider a quantum circuit on a set $Q$ of qubits, called circuit qubits, and a sequence of $N>0$ layers $L^1, \ldots, L^{N}$.
Each layer $L^t\subset Q \times Q$ consists of disjoint pairs of circuit qubits representing the two-qubit gates in the circuit.
Moreover, we consider a directed graph $H=(V,A)$ with vertices $V$, $|V|\geq |Q|$, and arc set $A$ representing the hardware qubits and hardware-native CX gates, respectively.
Here, $A$ is assumed to be symmetric, that is if $(i,j)\in A$, then also $(j,i) \in A$.
For $i,j\in V$, let $d_H(i,j)$ denote the length of a shortest path connecting $i$ and $j$ in $H$.
We introduce binary variables with the following interpretations.
Variable  $ x^t_{q, i, j} \in \set{0,1} $
takes value~$1$ if qubit $ q \in Q $
changes its assignment from node $ i \in V $ to node $ j \in V $
between layer $t$ and $ t + 1 $ from $ \set{1, \ldots, N} $,
and~$0$ otherwise.
The auxiliary variable $ w_{q, i}^ t \in \set{0,1} $
takes value~$1$ when qubit $ q \in Q $
is located at node $ i \in V $
in layer $ t \in \set{1, \ldots, N} $,
and~$0$ otherwise.
Further, the auxiliary variable $ z^t_{(p ,q), (i, j)} \in \set{0,1} $ takes value~$1$ when gate $ (p, q) \in L^t $
is performed along edge $ (i, j) $ and~$0$ otherwise.
With these notions, the TAP is modeled
by the following quadratic binary program, see Ref.~\cite{WBLW_2023}:
\begin{mini!}
	{x}{\displaystyle\sum_{t = 1}^{L-1}\displaystyle\sum_{q\in Q}\displaystyle\sum_{i,j\in V\times V} d_H(i,j) x^t_{q,i,j}\label{eq:Objective}}
	{\label{eq:edgemodel}}{}
	\addConstraint{w^t_{q,i}}			{=\displaystyle\sum_{j\in V}x^t_{q,i,j}\label{eq:flowout} }{\quad\forall 1\leq t \leq N-1,\forall i \in V,\forall q \in Q}
	\addConstraint{w^t_{q,i}}			{=\displaystyle\sum_{j\in V}x^{t-1}_{q,j,i}\label{eq:flowin} }{\quad\forall 2 \leq t \leq N,\forall i \in V,\forall q \in Q}
 	\addConstraint{\displaystyle\sum_{i\in V} w_{q,i}^t}{= 1\label{eq:bij1}}{\quad\forall 1\leq t\leq N,\forall q\in Q}
	\addConstraint{\displaystyle\sum_{q\in Q} w_{q,i}^t}{\leq1 \label{eq:bij2}}{\quad\forall 1\leq t\leq N,\forall i\in V}
	\addConstraint{\displaystyle\sum_{(i,j)\in A_H} z^t_{(p,q),(i,j)}}{=1\label{eq:gates} }{\quad\forall 1\leq t\leq N,\forall (p,q)\in L^t}
	\addConstraint{ z^t_{(p,q),(i,j)}}{= w_{p,i}^t\cdot w_{q,j}^t\label{eq:zdef}}{\quad\forall 1\leq t\leq N,\forall (p,q)\in L^t,\forall (i,j)\in A_H}
  	\addConstraint{\displaystyle\sum_{q\in Q} w_{q,i}^t}{=\sum_{q\in Q} w_{q,i}^1 \label{eq:subg}}{\quad\forall 2\leq t\leq N,\forall i\in V}
	\addConstraint{ w_{q,i}^t}{\in \{0,1\}\label{eq:binw}}{\quad\forall 1\leq t\leq N,\forall q\in Q,\forall i\in V}
	\addConstraint{ x^t_{q,i,j}}{\in \{0,1\}\label{eq:binx}}{\quad\forall 1\leq t< N,\forall q\in Q,\forall (i,j)\in V\times V}
	\addConstraint{ z^t_{(p,q),(i,j)}}{\in \{0,1\}\label{eq:binz}}{\quad\forall 1\leq t\leq N,\forall (p,q)\in L^t,\forall (i,j)\in A_H.}
\end{mini!}
\\[-\baselineskip]
Informally speaking, a feasible solution to Model~\eqref{eq:edgemodel} gives a mapping from circuit qubits to hardware qubits for each layer such that 
circuit qubits involved in two-qubit gates are always mapped to neighboring hardware qubits.
An optimal solution minimizes the total distance logical qubits move on the hardware graph $H$, which gives rise to a lower bound on the number of swaps required~\cite{Miltzow_2016}.
Constraints~\eqref{eq:flowout} and~\eqref{eq:flowin} ensure circuit qubit conservation.
Constraints~\eqref{eq:bij1} ensure that every circuit qubit is allocated to exactly one hardware qubit whereas
Constraints~\eqref{eq:bij2} ensure that every hardware qubit holds at most one circuit qubit.
Via Constraints~\eqref{eq:gates},
we enforce that every gate is implemented.
Constraints~\eqref{eq:zdef} demand that a gate
is implemented along an arc if and only if circuit qubits
are located at the hardware qubits of the arc.
Constraints~\eqref{eq:subg} enforce that the subgraph 
at which circuit qubits are located
is fixed for all time steps.
However, the particular choice of this subgraph remains subject to optimization.
We note that Constraints~\eqref{eq:zdef} contain quadratic expressions.
They need to be linearized to employ branch-and-cut solvers.
In Ref.~\cite{WBLW_2023} they are linearized by a standard McCormick approach replacing them by
\begin{align}
	\label{eq:zmccormick}
	z^t_{(p, q), (i, j)} &\leq w_{p, i}^t\notag\\
	z^t_{(p, q), (i, j)} &\leq w_{q, j}^t\notag\\
	z^t_{(p, q), (i, j)} &\geq w_{p, i}^t + w_{q, j}^t - 1.
\end{align}

We now derive an improved linearization by first considering the polytope defined by the convex hull of all binary-valued points satisfying \eqref{eq:zdef}.
This forms the well-known \emph{Boolean Quadric Polytope} (BQP), first introduced by Padberg~\cite{Padberg1989}
and extensively studied in the literature~\cite{GupteNote16,Gupte2020,AB2020,SRIPRATAK2022}.
For a comprehensive review we refer to Ref.~\cite{Deza1997}.
For arbitrary BQPs, no polynomial-sized linear description exists.
However, when additionally considering the single-choice constraints \eqref{eq:bij1} and \eqref{eq:gates} from the TAP model, we derive a linear description of the associated polytope
\begin{align}
	P \coloneqq \conv{\set{(w^t_{p,i},w^t_{q,j},z^t_{p,q,i,j}) \in \set{0,1}^{{\sum_t |L^t|\cdot |V|+\sum_t |L^t|\cdot |V| + \sum_t |L^t|\cdot |A_H|}} \mid \eqref{eq:zdef},\eqref{eq:gates},\eqref{eq:bij1}}}\ .
\end{align}
A complete linear description of $P$ improves the linear relaxation bound of Model~\eqref{eq:edgemodel}.
We first observe that $P$ is the cross product of several smaller-dimensional polytopes.
Clearly, when considering two different layers $t$ and $t'$, the set of variables occurring in \eqref{eq:zdef}, \eqref{eq:gates} and \eqref{eq:bij1} for $t$ and $t'$ are disjoint.
The same holds true when considering different $(p,q),(r,s) \in L^t$ for fixed $t$
since gates in the same layer act on disjoint qubits, i.e., $(p,q),(r,s) \in L^t$ implies $\set{p,q}\cap\set{r,s} = \emptyset$.
Thus, it follows
\begin{align}\label{eq:bqp_disjoint}
P = \bigotimes_{t=1}^L\bigotimes_{(p,q)\in L^t} P^t_{p,q}\,
\end{align}
where
\begin{align}
	P^t_{p,q} \coloneqq \conv{\set{(w^t_{p,i},w^t_{q,j},z^t_{p,q,i,j}) \in \set{0,1}^{{|V|+|V| + |A_H|}}\mid  \eqref{eq:zdef},\eqref{eq:gates},\eqref{eq:bij1}}}
\end{align}
is the BQP with single-choice constraints for fixed $t$ and fixed $(p,q) \in L^t$.
Thus, without loss of generality we consider a fixed $t$ and a fixed $(p,q) \in L^t$ and derive a complete description of $P^t_{p,q}$ in terms of linear inequalities.
Because of \eqref{eq:bqp_disjoint}, this yields a complete linear description of $P$.

For our further derivation, we first associate a graph $G$ with $P^t_{p,q}$.
$G$ has a vertex for every $w$ variable.
Edges in $G$ correspond to $z$ variables, i.e., they connect vertices whose corresponding $w$ variables occur in a bilinear term in the right hand side of \eqref{eq:zdef}.
Furthermore, for each right hand side $w^t_{p,i}\cdot w^t_{q,j}$ in \eqref{eq:zdef}, we have $p\neq q$.
This is because a two-qubit gate acts on two different qubits, i.e., $(p,q)\in L^t$ implies  $p\neq q$.
From this it directly follows that $G$ is bipartite with vertex partitions $\set{(p,i)\mid i \in V}$ and $\set{(q,j)\mid j \in V}$.
The edges of $G$ are given by $\set{\set{(p,i),(q,j)} \mid (i,j)\in A_H}$,
see Fig.~\ref{fig:G} for an illustration.
By definition of $A_H$ it follows that $G_{P}$ is symmetric in the sense that for every edge $\set{(p,i),(q,j)}$ there is an edge $\set{(p,j),(q,i)}$.
Summarizing, we can write $G$ as $(V_1 \dot{\cup} V_2,E)$ with $V_1 =\set{(q,i)\mid i \in V}$, $V_2 =\set{(p,j)\mid j \in V}$ and $E=\set{\set{(p,i),(q,j)} \mid (i,j)\in A_H}$.

\begin{figure}
	\centering
	\subfloat[]{
	\includegraphics[height=2.5cm]{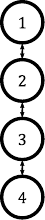}
	\label{fig:H}
	}\qquad
	\subfloat[]{
	\includegraphics[height=2.5cm]{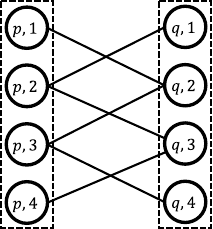}
	\label{fig:G_mc2}
	}
	\caption{(a)~Example for the directed hardware graph~$H=(V,A)$, representing four linearly connected qubits.
        (b)~Example for the bipartite graph~$G$ arising from $H$ associated with the polytope $P^t_{p,q}$
		for fixed $t$ and fixed $(p,q)\in L^t$.
		Indicated are the choice constraints \eqref{eq:bij1} for $p$ and $q$ (dashed).}
	\label{fig:G}
\end{figure}

We now derive a complete linear description of $P^t_{p,q}$ as a special case of a more general result for arbitrary bipartite graphs.
To this end, let $G=(X\dot{\cup}  Y,E)$ be a bipartite graph.
We consider the set of equations
\begin{align}
	z_{ij} &= x_i \cdot y_j \qquad \forall \set{i,j}\in E \label{eq:quadz}\\
	\sum_{i\in X}x_i &= 1 \label{eq:mcx}\\
	\sum_{j\in Y}y_j &= 1 \label{eq:mcy}\\
	\sum_{ij\in E} z_{ij} &= 1 \label{eq:mcz}\,,
\end{align}
see Fig.~\ref{fig:G_abp} for an illustration. Here, the shorthand $ij$ refers to the edge $\set{i,j}\in E$.
\begin{figure}
	\centering
	\subfloat[]{
		\includegraphics[height=4cm]{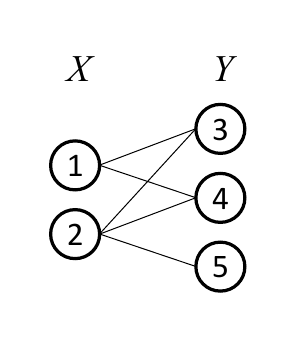}
		\label{fig:G_abp}
	}
	\subfloat[]{
		\includegraphics[height=4cm]{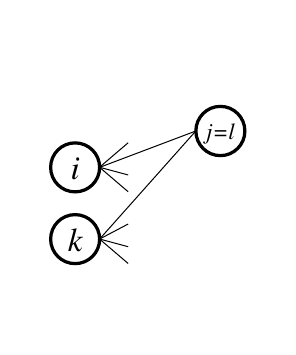}
		\label{fig:G_abp_case1}
	}
	\subfloat[]{
		\includegraphics[height=4cm]{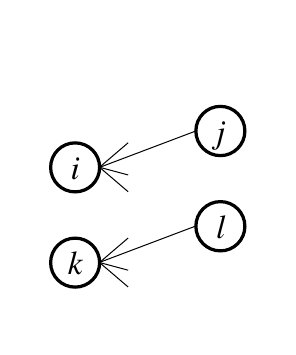}
		\label{fig:G_abp_case2}
	}
	
	\caption{Illustration for the setting of Lemma \ref{lem:BQPMCXYZ}.
		(a) An arbitrary bipartite graph for the definition of the polytope $P_{X,Y,Z}(G)$.
 		(b) Fractional variables in the first case in the proof. 
 		(c)~Fractional variables in the second case in the proof.
 		}
	\label{fig:G_lem}
\end{figure}
We study the associated polytope \[P_{X,Y,Z}(G)\coloneqq\conv{\{\set{0,1}^{|X|+|Y|+|E|}\mid \eqref{eq:quadz}, \eqref{eq:mcz}\}}\ .\]
We note that \eqref{eq:mcx} and \eqref{eq:mcy} are implied by \eqref{eq:quadz} and \eqref{eq:mcz}.
As we will show, the following equalities are valid for $P_{X,Y,Z}(G)$,
\begin{align}
		\sum_{j \in N(i)} z_{ij} &= x_{i}\quad \forall i \in X \label{eq:eqx2}\\
		\sum_{i \in N(j)} z_{ij} &= y_{j}\quad \forall j \in Y\ .\label{eq:eqy2}
\end{align}
Here, $N(i)$ denotes the neighborhood of $i\in G$.
We remark that applying the well-known Reformulation-Linearization Technique (see Ref.~\cite{Sherali1992ANR}) to \eqref{eq:mcx}, \eqref{eq:mcy} and variable bounds only yields
a \qm{$\leq$} in \eqref{eq:eqx2} and \eqref{eq:eqy2} and is thus not sufficient to derive the considered formulation.
In Ref.~\cite{GupteNote16} the authors show that $\eqref{eq:mcx}, \eqref{eq:eqx2}, \eqref{eq:eqy2}$
suffice to describe $P_{X,Y,Z}(G)$ for the complete bipartite case $G=K^{m,n}$.
Here, we prove it for arbitrary bipartite graphs.
\begin{lemma}\label{lem:BQPMCXYZ}
Let $G=(X\dot{\cup}  Y,E)$ be a bipartite graph.\\
Let $\tilde{P}:= \{[0,1]^{|X|+|Y|+|E|}\mid  \eqref{eq:mcx}, \eqref{eq:eqx2}, \eqref{eq:eqy2}\}$.
Then it holds $\tilde{P}=P_{X,Y,Z}(G)$.
\end{lemma}
\begin{proof}
	First, we show $P_{X,Y,Z}(G)\subseteq \tilde{P}$. 
    To this end, we need to show the validity of \eqref{eq:mcx}, \eqref{eq:eqx2} and \eqref{eq:eqy2} for $P_{X,Y,Z}(G)$.
	As already noted earlier, \eqref{eq:mcx} directly follows from \eqref{eq:quadz} and \eqref{eq:mcz}.
	Let $(x,y,z)\in\set{0,1}^{|X|+|Y|+|E|}$ be a vertex of $P_{X,Y,Z}(G)$.
	Consider the right-hand-side of \eqref{eq:eqx2}.
	In the case $x_i=0$ it directly follows from \eqref{eq:quadz} that $\sum_{j \in N(i)} z_{ij}=0$.
	On the other hand, if $x_i=1$, we have $\sum_{j \in N(i)} z_{ij}=1$ by constraints \eqref{eq:mcx}, \eqref{eq:mcz} \tb{and~\eqref{eq:quadz}}.
    Thus, \eqref{eq:eqx2} is valid for $P_{X,Y,Z}(G)$.
	The same argument applies for showing the validity of \eqref{eq:eqy2}.
	
	Now we show $\tilde{P}\subseteq P_{X,Y,Z}(G)$.
	We do so by showing that all vertices of $\tilde{P}$ are in $P_{X,Y,Z}(G)$.
	First, we note that $y$-choice \eqref{eq:mcy} and $z$-choice \eqref{eq:mcz} are valid for $\tilde{P}$:
	since $G$ is bipartite, we have 
	\[\sum_{ij \in E} z_{ij}=\sum_{i\in X}\sum_{j \in N(i)} z_{ij} \stackrel{\eqref{eq:eqx2}}{=}\sum_{i \in X} x_i \stackrel{\eqref{eq:mcx}}{=} 1\ .\]
	Analogously, \[1=\sum_{ij \in E} z_{ij}=\sum_{j\in Y}\sum_{i \in N(j)} z_{ij} \stackrel{\eqref{eq:eqy2}}{=}\sum_{j \in Y} y_j\ .\]
	Let $(x,y,z)\in \set{0,1}^{|X|+|Y|+|E|}$ be an integer vertex of $\tilde{P}$.
	Then, by $x$-choice \eqref{eq:mcx} there is exactly one $i\in X$ such that $x_i=1$.
	Analogously by $y$-choice \eqref{eq:mcy}, there is exactly one $j\in Y$ such that $y_j = 1$ and by $z$-choice \eqref{eq:mcz} exactly one $kl\in E$ such that $z_{kl}=1$.
	Assume $k\neq i$. Then, $\sum_{j \in N(i)} z_{ij}=0$, \tb{which is} a contradiction to \eqref{eq:eqx2} for $i$. Thus, $k=i$.
	By the same argument for $j$ we conclude $l=j$.
	Altogether, $(x,y,z)$ satisfies \eqref{eq:quadz} and thus $(x,y,z)\in P_{X,Y,Z}(G)$.
	
	Next, we show that all vertices of $\tilde{P}$ are integer.
	Let $(x,y,z)\in [0,1]^{|X|+|Y|+|E|}$ be a fractional point in $\tilde{P}$.
	Then at least one $x_i$ or one $y_i$ is fractional, otherwise $(x,y,z)$ would be integer.
	Without loss of generality,\ let $x_i$ be fractional.
	By $x$-choice $\eqref{eq:mcx}$ we know there is at least one other fractional $x_k$ with $k\neq i$.
	Furthermore, by Equality \eqref{eq:eqx2} for $i$, it follows that $\sum_{j \in N(i)} z_{ij}$ is fractional.
	This means at least one $z_{ij}$ for $j \in N(i)$ is fractional.
	The same argument gives a fractional $z_{kl}$ for $l \in N(k)$.
	Moreover, by Equality \eqref{eq:eqy2} for $j$ it follows $y_j>0$, since $z_{ij} > 0$.
	Analogously, $y_l>0$.
	
	Case 1: $j=l$, see Fig.~\ref{fig:G_abp_case1}.
%	Thus, $y_j=1$ and by $y$-multiple choice $\eqref{eq:mcy}$ $l=y$.
    \tb{We define a vector $a\in\set{-1,0,1}^{|X|+|Y|+|E|}$ by	$a = e_i + e_{ij} - e_k - e_{kl}$, where $e_m\in \set{0,1}^{|X|+|Y|+|E|}$ denotes the $m$-th unit vector.}
	Then there is an $\varepsilon > 0$ such that $(x,y,z) \pm \varepsilon a \in \tilde{P}$.
	
	Case 2: $j\neq l$, see Fig.~\ref{fig:G_abp_case2}. Since $y_j>0$ and $y_l>0$, we have $y_j\neq 1 \neq y_l$.
    \tb{We define $a\in\set{-1,0,1}^{|X|+|Y|+|E|}$ by	$a = e_i + e_j + e_{ij} - e_k -e_l - e_{kl}$.}
	Then, there is an $\varepsilon > 0$ such that $(x,y,z) \pm \varepsilon a \in \tilde{P}$.
	
	In either case, $(x,y,z)$ is not a vertex.
	Thus, $\tilde{P}$ has only integer vertices.
\end{proof}
Applying Lemma~\ref{lem:BQPMCXYZ} to the TAP model~\eqref{eq:edgemodel}, we replace Constraints~\eqref{eq:bij1} and~\eqref{eq:zdef} by
\begin{align}
	\sum_{j \in N(i)} z^t_{(p,q)(i,j)} &= w^t_{p,i}\quad\forall 1\leq t\leq L,\forall (p,q)\in L^t, \forall i \in V\\
	\sum_{i \in N(j)} z^t_{(p,q)(i,j)} &= w^t_{q,j}\quad\forall 1\leq t\leq L,\forall (p,q)\in L^t, \forall j \in V\ .
\end{align}
In our computational studies it turned out that this linearization results in a runtime improvement by a factor of two on average compared to the McCormick relaxation \eqref{eq:zmccormick}.
Therefore, we use this from now on.

\subsection{Noise Suppression via Qubit Routing}
Next, we generalize Model~\eqref{eq:edgemodel} to also suppress noise. 
We simultaneously aim for a small number of swap gates and noise suppression. 
These two, possibly conflicting, criteria lead to a multi-criteria optimization problem that we solve via a standard single-objective problem with an objective built from the weighted sum of both criteria. 
Therefore, noise data is incorporated in the basic Model~\eqref{eq:edgemodel} by extending the cost function~\eqref{eq:Objective} with additional terms.
Here, we first define a weighting factor $0\leq \lambda \leq 1$ which interpolates between only considering swap count ($\lambda = 0$) and only considering noise robustness ($\lambda = 1$).

Moreover, we allow additional costs $E_i$ for hardware qubits $i\in V$, costs $E_{(i,j)}$ for native CX $(i,j)\in A$ and costs $E_{(i,j),k}$ for CX-SQ crosstalk between $(i,j) \in A$ and $k\in N((i,j))$, where $N((i,j))\coloneqq (N(i)\setminus \{ j \} ) \cup (N(j)\setminus \{ i \} )$ is the neighborhood of arc $(i,j)$.
These additional costs  quantify the noise level of the associate hardware component.
For a physical qubit $i\in V$, we set $E_i$ as the average of readout error rate and single-qubit gate error rate.
Similarly, for a native CX gate $(i,j)\in A$, $E_{(i,j)}$ is defined as the average two-qubit error rate reported by the backend.
Taking the data provided by standard calibration routines for the quantum device used in this work,
the values of $E_{(i,j)}$ and $E_i$ are typically in the order of 1 \%.
For CX-SQ crosstalk between CX $(i,j)$ and qubit $k$, $(i,j) \in A,\ k\in N((i,j))$, let $r^{k}_{(i,j)}$ be the corresponding ERR, see Sec.~\ref{sec:srb}.
Then, we set
\begin{align}
E_{(i,j),k} \coloneqq \max \left\{0,\big( r^{k}_{(i,j)} - 1 \big) E_k \right\}\ .
\end{align}
Thus, if no crosstalk exists, \ie $r^{k}_{(i,j)}=1$, then $E_{(i,j),k}=0$.
Taking the experimental data from Sec.~\ref{sec:srb}, typical values of $E_{(i,j),k}$ lie between zero and 35 \%.
Having defined the individual costs, the overall cost function now reads
\begin{align}\label{eq:lambda}
c \coloneqq (1-\lambda)\cdot c_\mathrm{swap} + \lambda \cdot c_\mathrm{noise}
\end{align}
where
\begin{align}
    c_\mathrm{swap} \coloneqq \sum_{t = 1}^{L-1}\sum_{q\in Q}\sum_{i,j\in V\times V} d_H(i,j) x^t_{q,i,j}
\end{align}
as before, and
\begin{subequations}
\begin{align}
    c_\mathrm{noise} &\coloneqq \sum_{i \in V} E_i  \sum_{t = 1}^{L} \sum_{q\in Q}w^t_{q,i} \label{eq:sq_pen}\\
    &+ \sum_{(i,j)\in A} E_{(i,j)} \sum_{t = 1}^{L} \sum_{(p,q)\in L^t}  z^t_{(p,q),(i,j)} \label{eq:cx_pen}\\
    &+ \sum_{(i,j)\in A} \sum_{k \in N((i,j))} E_{(i,j),k} \sum_{q \in Q}  \sum_{t = 1}^{L} \sum_{(r,s)\in L^t} w^t_{q,k} \cdot z^t_{(r,s),(i,j)} \ . \label{eq:xt_pen}
\end{align}
\end{subequations}
Terms \eqref{eq:sq_pen} and \eqref{eq:cx_pen} penalize the individual use of single qubit $i$ and CX gate $(i,j)$ with a penalty factor of $E_i$ and $E_{(i,j)}$, respectively.
The term \eqref{eq:xt_pen} gives an additional penalty of $E_{(i,j),k}$ if both CX $(i,j)$ and qubit $k$ are simultaneously used in layer $t$.
Here, we remark that also CX-CX crosstalk is penalized implicitly by \eqref{eq:xt_pen}.
Simultaneous use of neighboring CX gates $(i,j)$ and $(k,l)$, where $k\in N(j)$, causes a penalty of $E_{(i,j),k}+E_{(k,l),j}$
since $z^t_{(p,q),(i,j)}=z^t_{(r,s),(k,l)}=1$ for some $(p,q),(r,s)\in L^t$ implies $w^t_{r,k}=w^t_{q,j}=1$.
Note, that \eqref{eq:xt_pen} is a quadratic expression.
It is linearized analogously to \eqref{eq:zmccormick} by introducing auxiliary variables and McCormick inequalities.

\paragraph{Dynamical Decoupling.}
After crosstalk aware routing, there might still exist gate-qubit pairs suffering from crosstalk.
To suppress remaining crosstalk as well as static ZZ crosstalk, we insert DD sequences on qubits during idle times~\cite{Tripathi_2022,Viola_1998}.

\section{Evaluation of Crosstalk Mitigation on QAOA}\label{sec:exp}
We now evaluate the noise-aware routing algorithm in the context of the Quantum Approximate Optimization Algorithm (QAOA).
QAOA is a well-known heuristic algorithm for a general class of optimization problems, originally proposed in Ref.~\cite{farhi_2014quantum}.
QAOA produces candidate solutions to the optimization problem by sampling from a circuit.
While there are many methods to error mitigate expectation values it is not yet known how to efficiently and cheaply error mitigate samples.
In QAOA, device noise can be compensated for by drawing more samples~\cite{Barron2023}.
However, if the noise is too strong, the sampling overhead becomes larger than an exponential exhaustive search of the solution space.
Therefore, it is crucial to reduce noise, such as crosstalk, in sampling based applications which is why we focus on QAOA.

We evaluate our noise-aware routing algorithm on instances of the Maximum Cut problem (MaxCut),
which is equivalent to quadratic unconstrained binary optimization (QUBO)~\cite{Hammer_1965, DESIMONE_1990}.
Given a graph $G=(V,E)$, MaxCut asks for a partition of the nodes such that the number of
edges intersecting the partitions is maximum.
The MaxCut problem is an archetypical NP-hard optimization problem~\cite{Karp_1972}.
It is intensively studied in classical computation~\cite{Hadlock_1975,Barahona_1989,Liers_2003,Liers_2012,Rehfeldt2023}
and is often examined as a benchmark for QAOA~\cite{Weidenfeller2022, Harrigan_2021,Lykov_2022,Tate_2023}.

When applied to MaxCut, QAOA prepares the state
\begin{align}
\ket{\boldsymbol{\beta},\boldsymbol{\gamma}} = \prod_{k=1}^p e^{-i\beta_kH_M} e^{-i\gamma_k H_P} \ket{+}^{\otimes n}
\end{align}
where $ H_P = \sum_{ij\in E} \sigma_i^z \sigma_j^z$ and $H_M=-\sum_{i=1}^n\sigma_i^x$ are the problem and mixing Hamiltonian, respectively.
The initial state $\ket{+}^{\otimes n}$ is the equally weighted superposition of all solutions and the ground state of $H_M$.
The complexity of the transpiled circuit, here defined as the number of gates, is controlled by and proportional to the hyper-parameter $p\in \N$, called \emph{depth}.
In the experiments, we choose $p=1,\dots,7$.
The parameters $\boldsymbol{\beta}=(\beta_1,\dots,\beta_p$) and $\boldsymbol{\gamma}=(\gamma_1,\dots,\gamma_p)$ are usually 
optimized in a classical feedback loop to minimize the expected cut size $\bra{{\beta},\boldsymbol{\gamma}}H_P\ket{{{\beta},\boldsymbol{\gamma}}}$.
Crucially, although QAOA is trained on an expectation value, solutions to the MaxCut problem are ultimately retrieved via sampling~\cite{Barron2023}.
Moreover, routing is required to execute QAOA since
its implementation applies two-qubit gates along the edges of the underlying MaxCut graph, which is usually not a subgraph of the hardware connectivity graph.

The approximation ratio of QAOA is an easily accessible performance metric.
Here, we define the approximation ratio as the expected cut size divided by the optimum value.
QAOA is a parameterized algorithm and its performance, i.e.~the expected cut size, heavily depends on the values of $\boldsymbol{\gamma}$ and $\boldsymbol{\beta}$.
To avoid any bias in expected cut size resulting from a sub-optimal choice of $\boldsymbol{\gamma}$ and $\boldsymbol{\beta}$ we calculate these parameters beforehand with a classical, noise-free simulation
and an off-the-shelf optimizer~\cite{Powell1994ADS}.
This allows us to focus on the effect of the routing method.

As MaxCut instances, we consider a 14 vertex line, a 10 vertex three-regular graph and a complete graph on 5 vertices.
Intuitively, problem graphs with high edge density require many swaps when routing onto sparse hardware architectures~\cite{Weidenfeller2022}.
Therefore, denser graphs generally result in noisier results which is why we study a line, a three-regular and a fully connected graph.
These graphs increase in edge density which causes the routing algorithms to insert more swaps.
For example, the hardware graph shown in Fig.~\ref{fig:cm_ehningen} contains several line subgraphs with 14 vertices.
Thus, the line instance requires no additional swap gates if the initial mapping is chosen to be an isomorphism between the MaxCut graph and one of the line subgraphs.
However, this is not the case for the three-regular graph and the complete graph.
Here, additional swap gates are necessary and a trade-off between swap count and noise level has to be achieved by the noise-aware routing algorithms.

\subsection{Evaluated Routing Methods}
We transpile the QAOA circuits to \emph{ibmq\_ehningen} by several established noise-aware routing methods and compare the results after execution.
The noise data required for noise-aware routing is taken from daily calibration routines and, in the case of CX-SQ crosstalk noise, is retrieved from the experiments described in Sec.~\ref{fig:srb_ex}, also executed on the same day as the QAOA circuits.
Our crosstalk measurements revealed that severe crosstalk is typically limited to a small subset of qubits, similar to Fig.~\ref{fig:hist_cxrb}.
Thus, we only take the ten largest ERR values and set the rest to one.
This reduces the number of quadratic terms in Equation~\eqref{eq:xt_pen} drastically since most $E_{(i,j),k}$ values are zero.
As DD sequences, we choose ten equally spaced X gates.
This choice is based on experience and studies from literature~\cite{Tripathi_2022,Viola_1998}.

Transpilation is performed by the following methods.
\begin{enumerate}
    \item[(M1)] Use the method described in Section~\ref{sec:mitigation} with $\lambda = 0.5$ in Equation~\eqref{eq:lambda}, labeled noise-aware token allocation problem (NATAP). The choice of $\lambda = 0.5$ is based on empirical studies, compare Figs.~\ref{fig:cx_vs_lambda_line}, \ref{fig:cx_vs_lambda_3reg}, \ref{fig:cx_vs_lambda_compl}.
    The binary linear program \eqref{eq:edgemodel} is solved via \emph{Gurobi}~\cite{gurobi} with a time limit of 900~s.
    This value is also based on empirical studies which showed that a near-optimal solution is usually found within the first 100~s of the solution process.
    For $p>1$, we employ the commutation of gates in QAOA and construct the routed circuits by repeating and alternatingly reversing the routed $p=1$ circuit.
    \tb{The source code to our method (M1) is published in \cite{tap_ijoc}.}

    \item[(M2)] The same as (M1), but with $\lambda = 0$, \ie we do not consider noise data and only minimize swap count, labeled TAP.
    Comparing (M1) to this method allows us to study the influence of incorporating noise data.

    \item[(M3)] Choose the best line with respect to the product of CX errors for an initial layout, as proposed in Refs.~\cite{Sack_2023} and~\cite{Ferrari_2022}.
    \tb{We determine the best line by simply enumerating all lines of the required length in the hardware graph.}
    If necessary, perform the SABRE heuristic for routing~\cite{Li2019} \tb{whose source code is available on-line~\cite{sabre}.}
    This method comes at low computational costs (runtimes in the order of 1~s), but considers noise only roughly via CX errors.
    In particular, no crosstalk errors are considered.
    
    \item[(M4)] The same as (M3), but with additional DD sequences inserted on idle qubits.
    This method allows us to investigate the influence of DD.
    
    \item[(M5)] Use the noise-adaptive layout method proposed in Ref.~\cite{Murali_2019} as an initial layout and SABRE as routing, if necessary.
    \tb{The source code for the layout is available on-line~\cite{noiseadaptive}.}
    Afterwards, the routed circuit is scheduled via the CX-CX crosstalk aware scheduling method proposed in Ref.~\cite{Murali_2020}.
    \tb{The source code for the scheduling is available on-line~\cite{xtsched}.}
    Here, CX-CX crosstalk is measured via SRB.
    While layout and routing is fast (order of 1~s), the scheduling method relies on solving a satisfiability problem, which  requires on the order of 1 minute to solve.
    Apart from (M1), this is the only method that accounts for crosstalk.
    However, (M5) mitigates crosstalk by delaying gates which increases decoherence.
        
    \item[(M6)] Use Tket's noise-adaptive layout and routing method~\cite{Sivarajah_2020}.
    \tb{The Tket source code is available on-line~\cite{tket}.}
    This method is fast (order of 1~s), but considers noise data only in the choice of an initial layout and does not consider crosstalk.
\end{enumerate}

We use IBM's SDK Qiskit to construct circuits and communicate with the quantum backend~\cite{Qiskit}.
We execute each transpiled circuit $100,000$ times on the quantum hardware and compute the achieved approximation ratio.
Additionally, we compute the approximation ratio retrieved from an ideal simulation
as well as the approximation ratio corresponding to the uniform distribution, which resembles a completely depolarized quantum computer.

\subsection{Computational Results}
As expected, the ideal approximation ratio monotonically increases with $p$ on all three instances,
while the approximation ratios returned from real hardware tend towards the approximation ratio of the uniform distribution for large $p$, see Figs.~\ref{fig:qaoa_14q_line_appr_vs_depth}, \ref{fig:qaoa_10q3reg_appr_vs_depth}, \ref{fig:qaoa_5q_compl_appr_vs_depth}.
However, when comparing results from the different routing methods, the method proposed in this work, (M1), achieves the highest approximation ratios across all instances and depths (blue lines in Figs.~\ref{fig:qaoa_14q_line_appr_vs_depth}, \ref{fig:qaoa_10q3reg_appr_vs_depth}, \ref{fig:qaoa_5q_compl_appr_vs_depth}).  
Compared to methods (M3) to (M6), this significant improvement could partially be attributed to the smaller CX count of (M1), compare Figs.~\ref{fig:cx_vs_lambda_line}, \ref{fig:cx_vs_lambda_3reg}, \ref{fig:cx_vs_lambda_compl}.
However, method (M2) uses as many CX gates as method (M1) on all instances and for all depths.
As a consequence, the remarkable improvement in approximation ratio of method (M1) can only be attributed to the noise data incorporation.
(M1) avoids single qubits and CX gates with high error rates as well as the simultaneous use of CX-SQ pairs with large crosstalk.
From this, we conclude that it is highly advantageous to include noise as another criterion besides gate count or depth when transpiling quantum algorithms. 

Furthermore, we observe that (M4) (red lines in Figs.~\ref{fig:qaoa_14q_line_appr_vs_depth}, \ref{fig:qaoa_10q3reg_appr_vs_depth}, \ref{fig:qaoa_5q_compl_appr_vs_depth}) achieves considerably larger approximation ratios than (M3) (green lines) on the line instance and the complete graph instance.
Since the only difference is DD, we conclude that DD is a useful tool to suppress noise.

Analyzing the line instance in more detail, we observe that method (M5) does not choose a line subgraph as initial mapping, see Fig.~\ref{fig:layout_14line}.
As a result, unlike the other methods, (M5) needs to insert swap gates which leads to a larger CX count as Fig.~\ref{fig:cx_vs_lambda_line} shows.
Although method (M5) considers noise, the additional swaps lead to a disadvantage in terms of approximation ratio when compared to the other methods, clearly visible in~Fig.~\ref{fig:qaoa_14q_line_appr_vs_depth}. 
On the contrary, methods (M2), (M3), (M4) and (M6) use as many CX gates as the proposed (M1) approach on the line instance, as seen in Fig.~\ref{fig:cx_vs_lambda_line}.
Here, the difference in approximation ratio can only be attributed to the different ways of noise data incorporation.
Method (M2) does not consider noise at all, while methods (M3) to (M6) consider noise data in the initial layout.
Crosstalk noise, however, is considered in layout and routing only by method (M1).
As a result, (M1) chooses the subgraph with the smallest crosstalk levels, see Fig.~\ref{fig:layout_14line}.
Since (M1) achieves the largest approximation ratio, we conclude that considering crosstalk in the transpilation is highly beneficial.
In particular, we see that the best line in terms of CX gate errors, used by methods (M3) and (M4), fails to match the results of (M1).
This is because the CX gates are benchmarked in isolation and the line with the best product of CX errors includes a high crosstalk triplet between $\text{CX}_{22,25}$ and qubit 24.
Regarding runtime, (M1) and (M2) took roughly 30~s during which the binary linear models where solved to global optimality.
This is somewhat larger but still comparable to the other methods which took between $\sim$ 1~s and 20~s.

On the three-regular instance, methods (M1) and (M2) run into the time limit of 900~s.
\tb{However, methods (M1) and (M2) found the best solution already after 64~s and 119~s, respectively.}
Remarkably, method (M1) is the only method delivering approximation ratios significantly larger than a completely depolarized quantum computer.
Here, (M1) chooses the subgraph with the smallest crosstalk consisting of qubits 12, 13, 14, 15, 16, 18, 19, 20, 22 and 25, see Fig~\ref{fig:layout_10reg}.
Furthermore, when examining the transpiled circuits (not shown), we see that (M1) does not map any circuit two-qubit gate to the hardware $\text{CX}_{12,15}$ gate, since this would trigger large crosstalk on qubit 13, which is also in the chosen subgraph.
Notably, methods (M3) to (M6) use significantly more swaps than (M1) and (M2) resulting in a large CX count in Fig.~\ref{fig:cx_vs_lambda_3reg}.
These results show a clear benefit from the additional time investment in an improved routing solution.
Moreover, compared to parameter training and queue waiting, several minutes of additional routing time are bearable.

Also for the fully connected instance, methods (M1) and (M2) run into the time limit of 900~s.
\tb{However, the IP solver found the best solution already after 90~s and 10~s, respectively.
This indicates that for larger instances we can stop the optimizer early without degrading the solution quality.
Indeed, solvers like Gurobi typically spend most of their time proving that the found solution is optimal.
}
Moreover, (M1) achieves a significantly better approximation ratio than all other methods \tb{up to $p=3$}.
\tb{For $p>3$, the large number of gates causes a depolarization such that any transpilation method returns the uniform distribution.}
\tb{(M1) improves upon the best existing method (M4) by over 20~\% at $p = 1$ where we measure the improvement relative to the interval between the approximation ratio of the uniform distribution and the ideal $p=1$ QAOA which are $0.83$ and $0.98$, respectively.}
Notably, although (M6) uses less swaps than (M1), see Fig.~\ref{fig:cx_vs_lambda_compl}, it returns an approximation ratio not better than random sampling.
This is because (M6) uses the hardware gates $\text{CX}_{22,25}$ and $\text{CX}_{25,24}$ which trigger a large crosstalk on qubit 24 and 23, respectively.
By contrast, (M1) does not use any large crosstalk triplets.
Indeed, (M6) considers single and two qubit gate errors as well as readout errors but ignores crosstalk and thus chooses a subgraph with higher crosstalk levels, see Fig.~\ref{fig:layout_5compl}.
This further stresses the importance of crosstalk incorporation and shows that trading additional swap gates for low crosstalk can be beneficial.
The novel approach (M1) successfully incorporates this trade-off.
\tb{Similarly, (M5) uses as many CX gates as (M1), see Fig.~\ref{fig:cx_vs_lambda_compl}, but yields a significantly worse approximation ratio.
In Fig.~\ref{fig:layout_5compl}, we observe that (M5) chooses a subgraph containing a large crosstalk triplet.
This is because the noise-adaptive layout method from Ref.~\cite{Murali_2019}, which is used by (M5), is insensitive to crosstalk errors.
The scheduling method used by (M5) delays gates which suffer from large crosstalk.
Still, (M5) yields an approximation ratio no better than random sampling.
From this, we conclude that crosstalk can be avoided more effectively via qubit routing than via scheduling.
In particular, if the number of large crosstalk triplets is relatively small as is the case in our study,
crosstalk can be avoided often without inserting additional swaps.}
\begin{figure}[p]
	\makebox[\textwidth][c]{
    \subfloat[]{\includegraphics[width=0.5\linewidth,trim={0cm 0cm 1cm 1cm},clip]{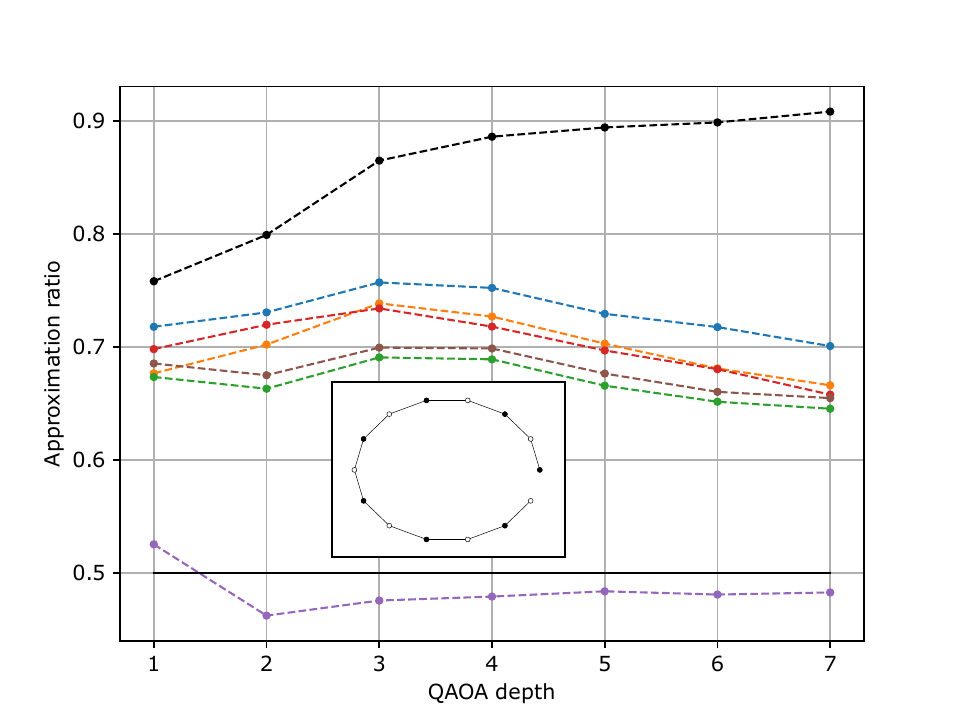}\label{fig:qaoa_14q_line_appr_vs_depth}}
    \subfloat[]{\includegraphics[width=0.3\linewidth,trim={3cm 0cm 3cm 3cm},clip]{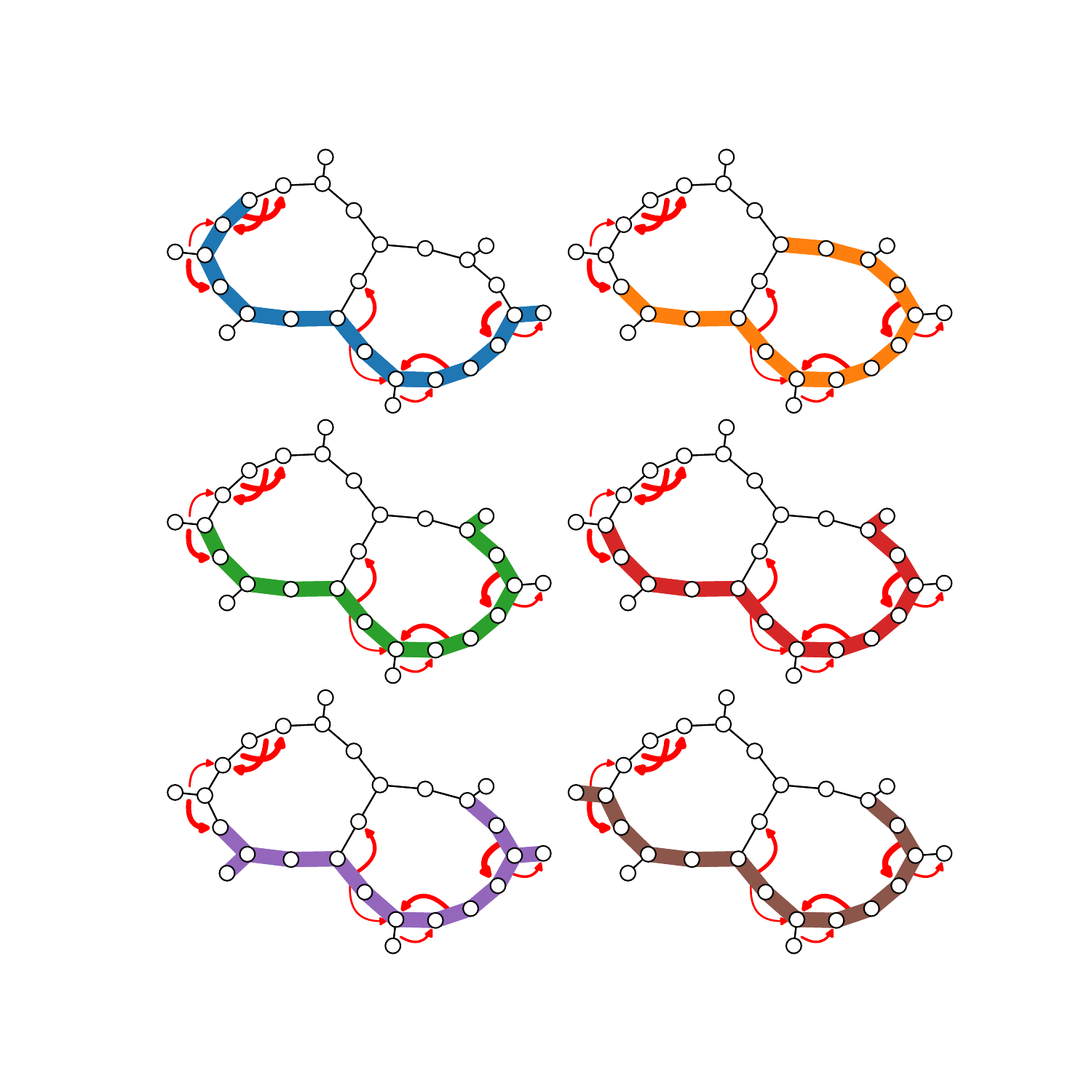}\label{fig:layout_14line}}
    \subfloat[]{\includegraphics[width=0.35\linewidth]{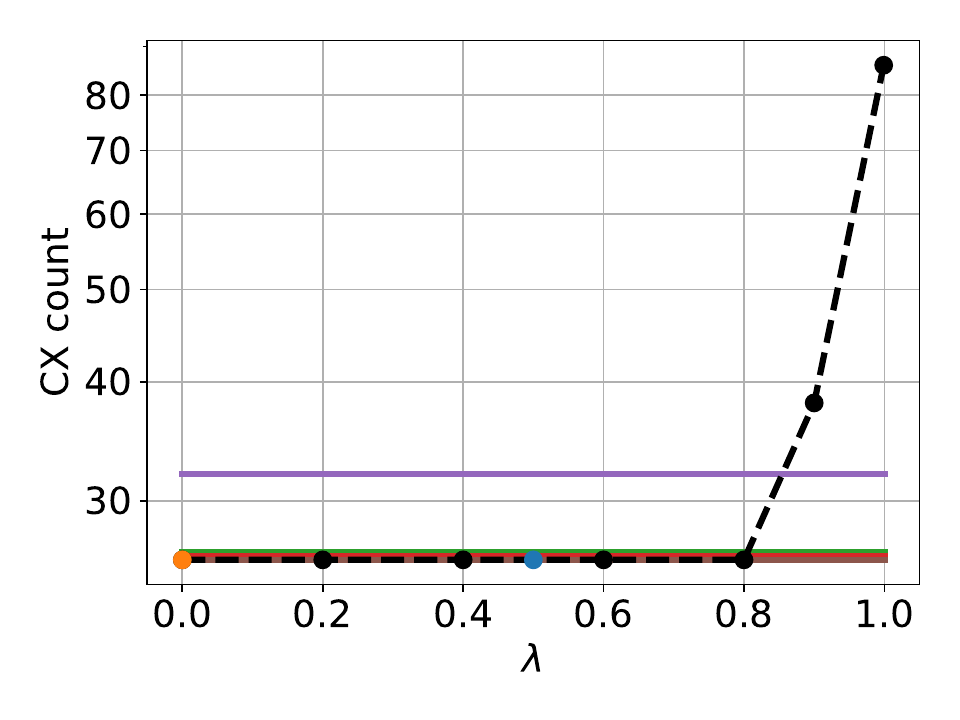}\label{fig:cx_vs_lambda_line}}
    }\\
    \makebox[\textwidth][c]{
    \subfloat[]{\includegraphics[width=0.5\linewidth,trim={0cm 0cm 1cm 1cm},clip]{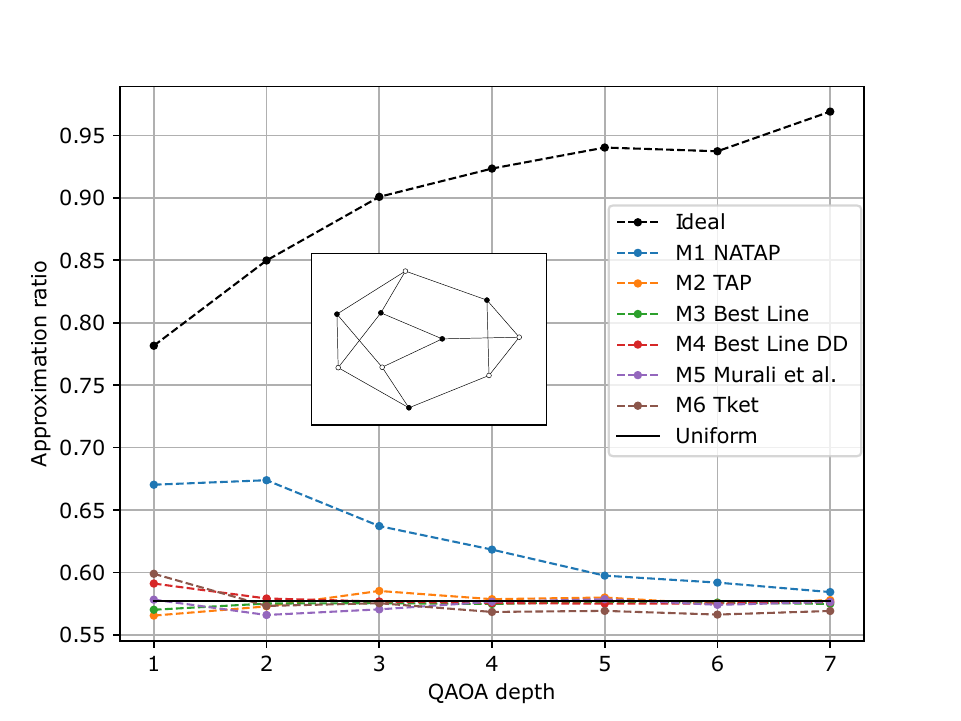}\label{fig:qaoa_10q3reg_appr_vs_depth}}
    \subfloat[]{\includegraphics[width=0.3\linewidth,trim={3cm 0cm 3cm 3cm},clip]{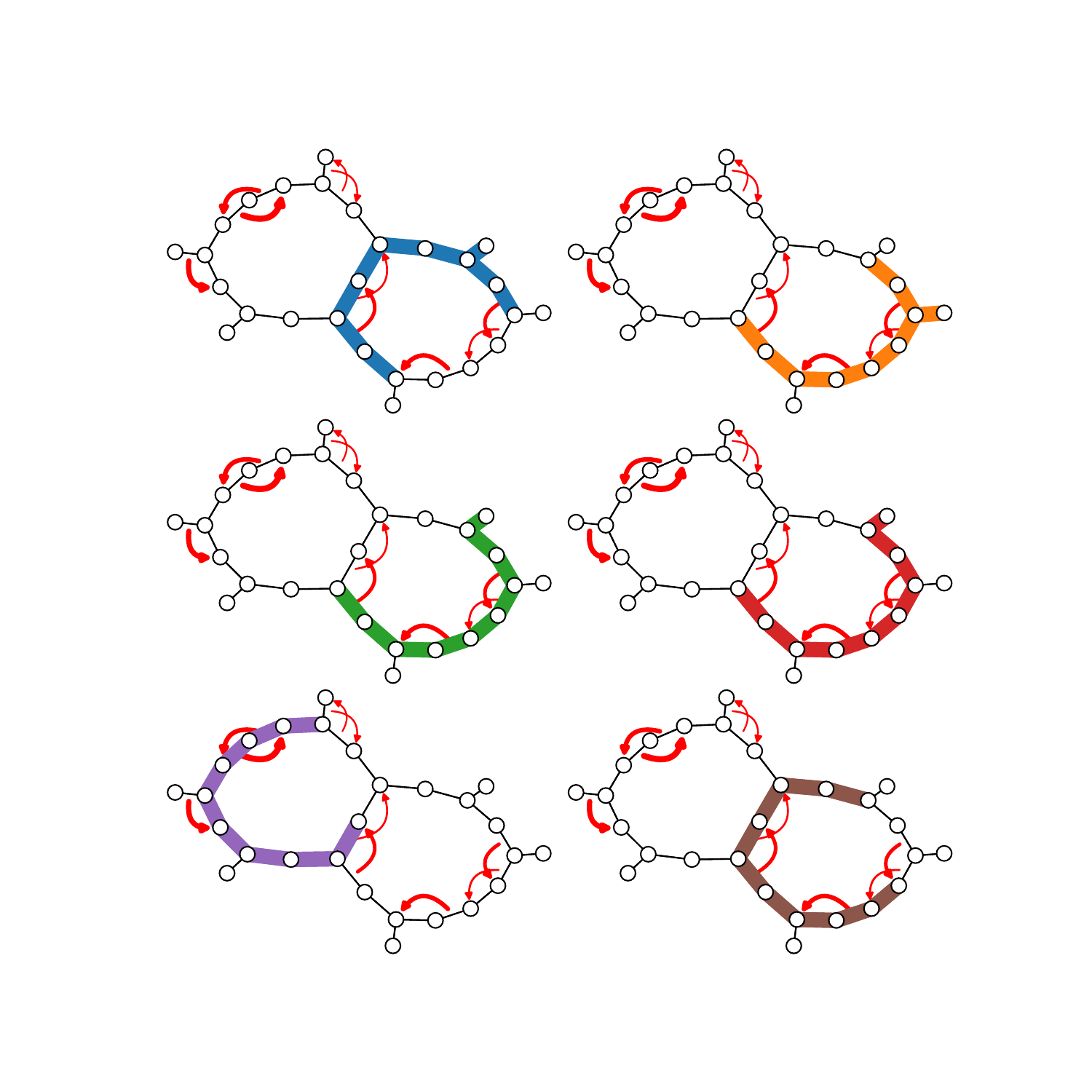}\label{fig:layout_10reg}}
    \subfloat[]{\includegraphics[width=0.35\linewidth]{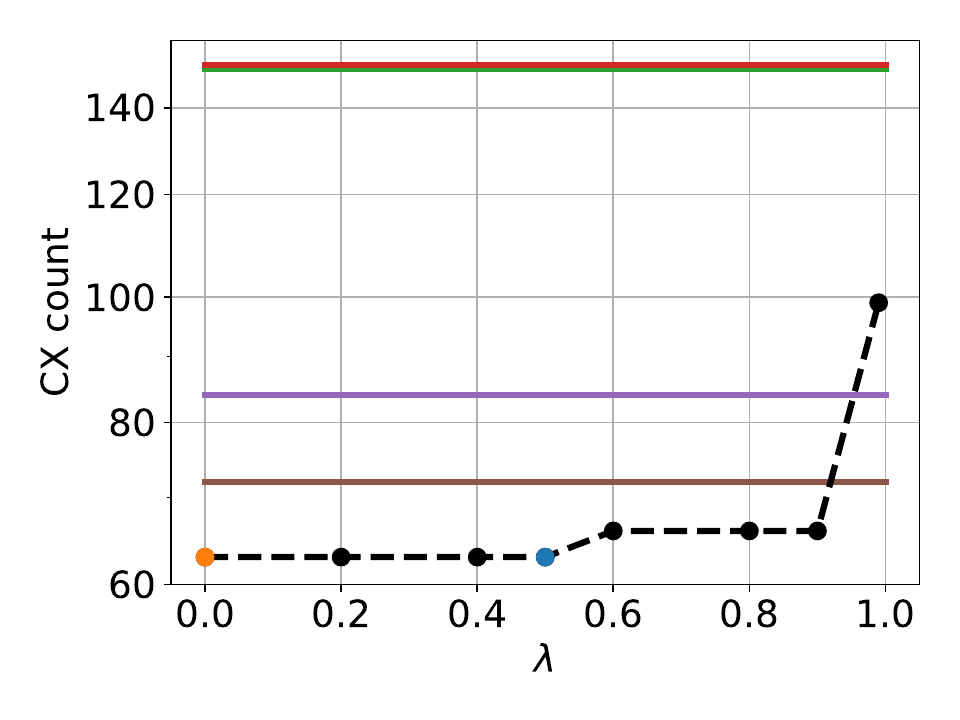}\label{fig:cx_vs_lambda_3reg}}
	}
     \makebox[\textwidth][c]{
    \subfloat[]{\includegraphics[width=0.5\linewidth,trim={0cm 0cm 1cm 1cm},clip]{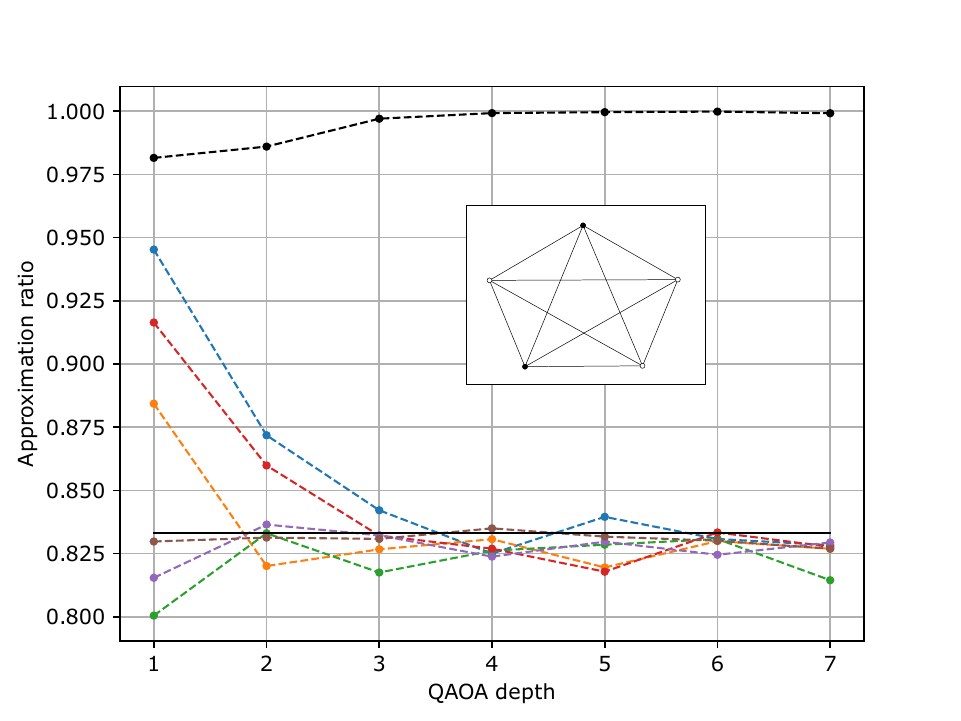}\label{fig:qaoa_5q_compl_appr_vs_depth}}
    \subfloat[]{\includegraphics[width=0.3\linewidth,trim={3cm 0cm 3cm 3cm},clip]{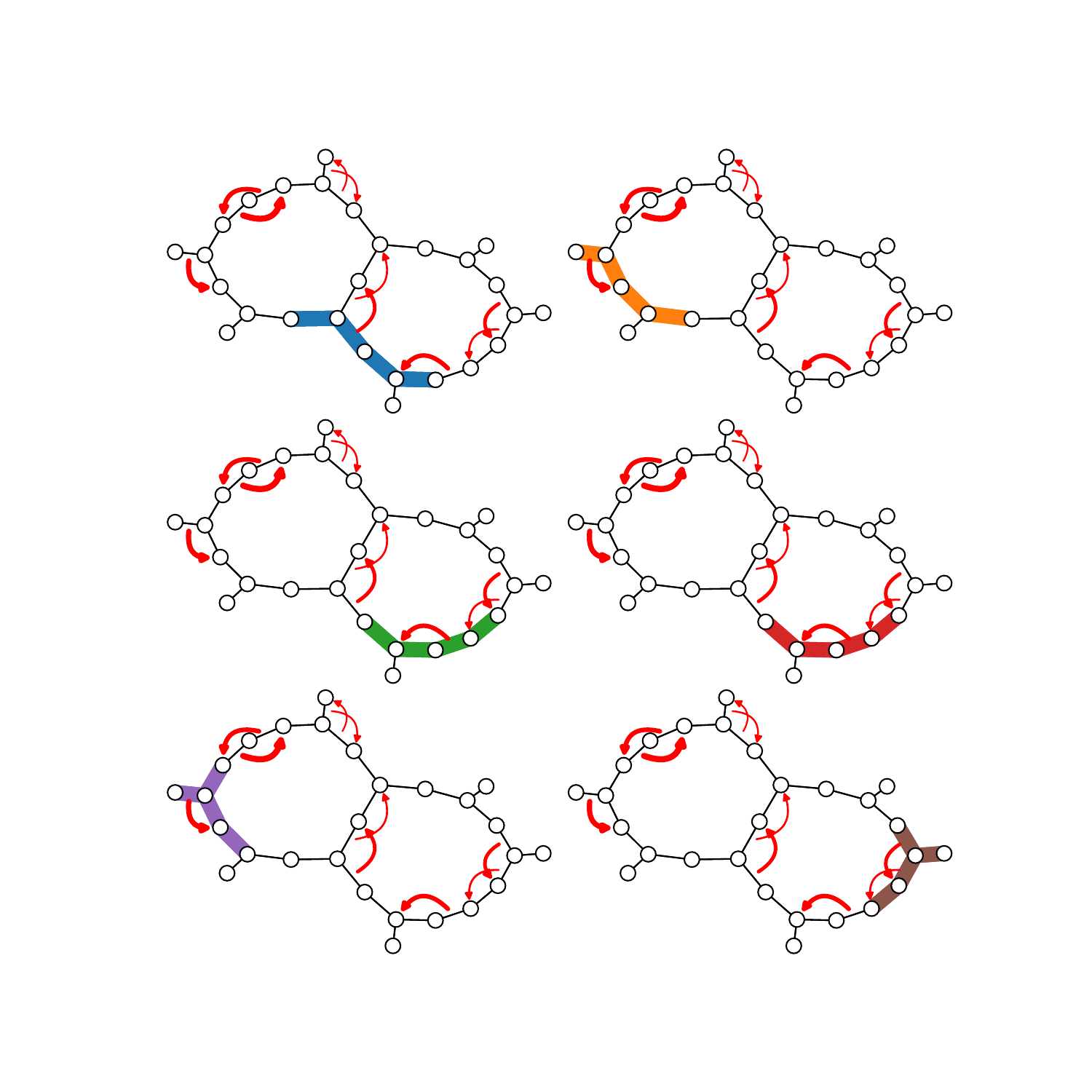}\label{fig:layout_5compl}}
    \subfloat[]{\includegraphics[width=0.35\linewidth]{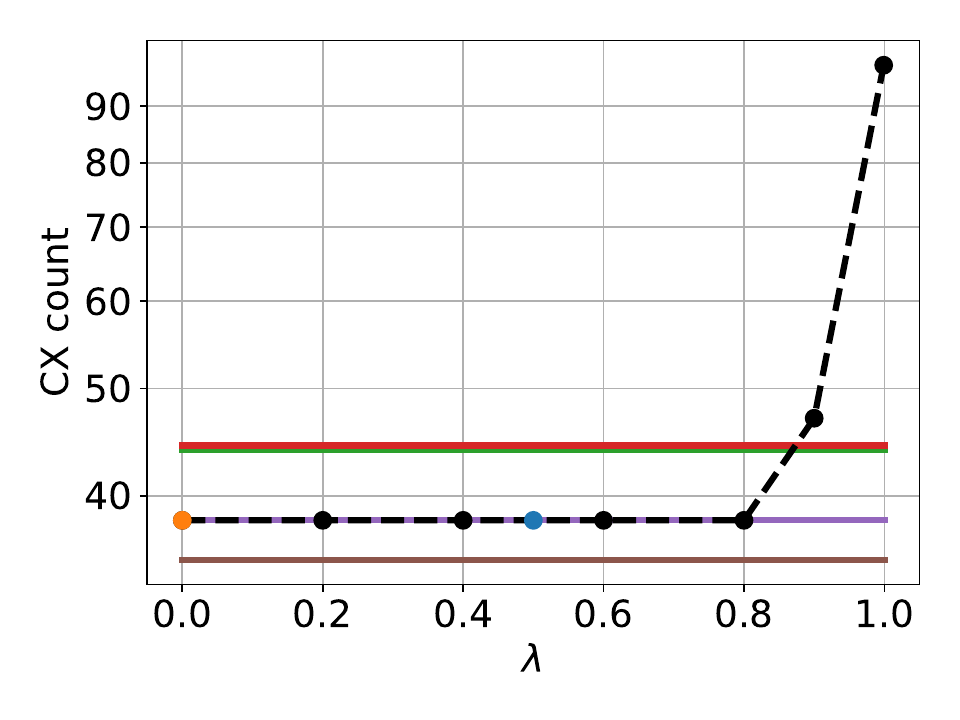}\label{fig:cx_vs_lambda_compl}}
	}
    \caption{Results for benchmarking different noise aware transpilation methods.
    (a), (d), (g) show the achieved approximation ratio versus the QAOA depth for the line, regular and fully connected instances (shown in inlays, where filled vertices mark an optimum cut). (b), (e), (h) highlight the subgraphs used for computation for each method. Additionally, high crosstalk is marked by red arrows, where the thickness is proportional to the crosstalk magnitude. Crosstalk magnitudes differ between instances since they were measured on different days. (c), (f), (i) show the total CX count in the transpiled circuits for QAOA depth $p=1$ versus the weighting factor $\lambda$.}
	\label{fig:qaoa_res}
\end{figure}

\tb{The lower CX count and noise level of method (M1) come at the cost of an increased transpilation time compared to the other methods.
We investigate this trade-off between transpilation time and performance gain by comparing our IP-based routing algorithm, used by method (M1), to the SABRE and Tket routing heuristics used by methods (M3)-(M5) and (M6), respectively.
To this end, we route QAOA circuits corresponding to MaxCut instances on three-regular graphs of increasing size to the coupling map of \emph{ibmq\_ehningen} and compare the number of inserted swaps and the resulting circuit depth.
We again allow Gurobi a maximum runtime of 900~s.
Our IP based algorithm inserts less swaps than SABRE and Tket on all instances, with an average reduction of 56~\% and 41~\%, respectively, see Fig.~\ref{fig:qaoa_res}(a).
The reduction in circuit depth, shown in Fig.~\ref{fig:depth_vs_size}, is even more significant with average reductions of 51~\% and 49~\% compared to SABRE and Tket, respectively.
Moreover, we observe in Fig.~\ref{fig:swaps_vs_size} that, for the smaller instances with $n\leq 14$ vertices, the best solution was already found within the first third of the total runtime.
For $n>14$ the best solution was found at the end of the allocated runtime.
Crucially, this probably sub-optimal solution has a lower depth and gate count than the other methods.
}
\begin{figure}
	\makebox[\textwidth][c]{
    \subfloat[]{\includegraphics[height=4.5 cm]{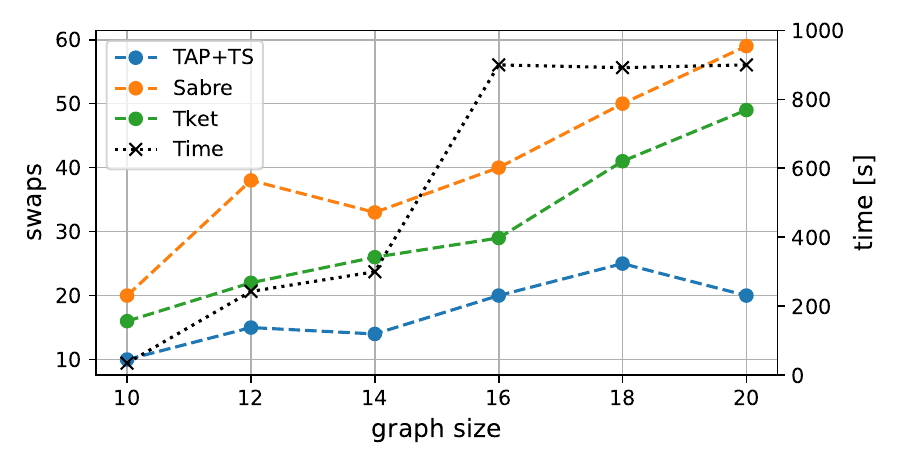}\label{fig:swaps_vs_size}}
    \subfloat[]{\includegraphics[height=4.5 cm]{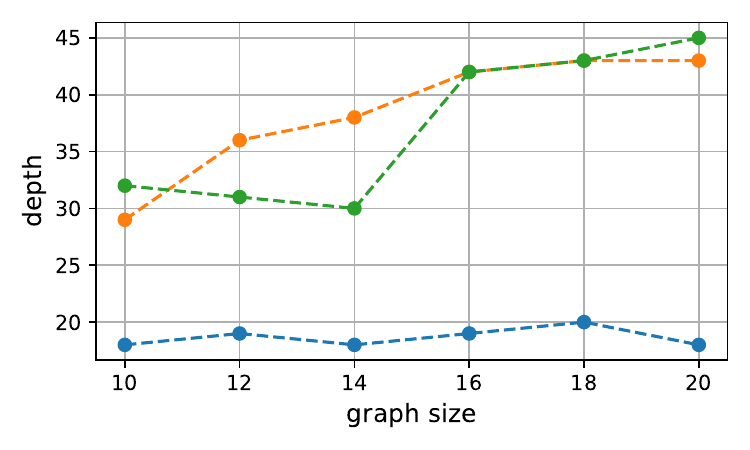}\label{fig:depth_vs_size}}
    }
    \caption{\tb{Comparison of our IP-based routing algorithm TAP+TS to the routing heuristics SABRE~\cite{Li2019} and Tket~\cite{Sivarajah_2020}.
    We route QAOA circuits corresponding to MaxCut instances on three-regular graphs of increasing size to the coupling map of \emph{ibmq\_ehningen}, shown in Fig.\ref{fig:cm_ehningen}.
    We allow a total runtime of 900~s.
    We compare the number of inserted swap gates in (a) and the resulting circuit depth in (b) which is the length of the critical path in the circuit including single- and two-qubit gates.
    On the second y-axis in (a), we report the time taken by the IP solver to find the returned solution.}}
	\label{fig:runtime}
\end{figure}

To summarize, noise-aware routing as performed by the proposed method (M1) improves quantum computation significantly by accounting for crosstalk errors in the transpilation process.
\tb{This improvement comes at the cost of increased transpilation time compared to other heuristics but our results show that the additional time investment considerably reduces noise.}
We conclude that crosstalk errors are crucial to consider since applications like QAOA drive multiple qubits simultaneously.
This further motivates the use of metrics such as layer fidelity for benchmarking quantum computers at scale~\cite{Mckay2023}.
Moreover, our (M1) results show that crosstalk can be mitigated without additional swap gates or delays if the routing is done with the method we propose.

\section{Conclusion}\label{sec:concl}
In this work, we first developed a simplified randomized benchmarking experiment to quantify crosstalk noise induced by two-qubit gates.
The simplified experiment does not rely on measurements of two-qubit gate error rates which significantly reduces the number of circuits to execute.
Furthermore, this allows us to replace random two-qubit gate sequences by a single stretched CX-pulse, simplifying compilation.
\tb{
This method is applicable to other architectures in which the multi-qubit gate is created by driving a corresponding control Hamiltonian with a pulse.
Indeed, one may simply drive this control for an extended time while RB is performed on the surrounding qubit(s).
Future work could therefore study the applicability of this method to other hardware platforms such as trapped ions and neutral atoms.}
\tb{Furthermore, our} experiments show that this method can replace standard experiments without degrading accuracy in cross-resonance based hardware.
\tb{Comparing our simplified protocol to other crosstalk characterization methods is a direction of future research.}
\tb{For example, randomized compiling was recently extended to measure gate-triggered crosstalk noise~\cite{Perrin_2023}.}
Moreover, developing crosstalk measurement methods which do not rely on random gates could simplify crosstalk measurement even further.

Our second contribution is an optimal experiment scheduling to minimize the overhead required for crosstalk characterization.
We model this task as a graph coloring problem and solve it to optimality for several relevant example architectures.
Our study reveals that heuristics leave room for improvement and that the overhead heavily increases with the density of the underlying graph.
\tb{Recently, Ref.~\cite{kattemölle2024edge} showed that an edge coloring of a sufficiently large subgraph of an infinite lattice
induces a proper coloring of the entire lattice.
Transferring this result to our vertex coloring problem is a promising direction of future research
since the graphs which we need to color are typically large subgraphs of infinite lattices.}
Furthermore, a detailed analysis of the time dependency of crosstalk errors may help further reduce the characterization overhead.
The measurements performed in the course of this work indicate that severe crosstalk is mostly limited to a fixed subset of qubits.
Thus, it may be sufficient to only characterize this subset.
Similarly, if crosstalk magnitude is relatively stable over time, the measurement frequency can be reduced.

The third contribution is a noise-aware routing method incorporating crosstalk data.
The routing algorithm builds upon previous work based on integer programming.
We improve the solution time of the underlying integer linear model by deriving a tighter linearization of quadratic constraints.
In this context, we derive a complete linear description of an associated Boolean Quadric Polytope on bipartite graphs with additional choice constraints.
This theoretical result has applications beyond this work.
Future work on polyhedral analysis can reduce runtime even further.
\tb{Crucially, we see that Gurobi rapidly finds high-quality solutions and spends most of its allocated time proving optimality.
Indeed, finding provably optimal circuits is not necessary when good-enough circuits suffice.
}
Noise data is included in the model via additional terms in the objective function.
To the best of our knowledge, the proposed method is the first to consider both standard noise data and crosstalk data.
We benchmark our method against five other routing algorithms with QAOA circuits which we execute on hardware.
These experiments reveal that \tb{our} crosstalk-aware routing significantly improves the measured results \tb{compared to other noise-aware} transpiler methods.
Interestingly, we observed that it can even be advantageous to trade additional swap gates for low noise.
Recently, Ref.~\cite{Sharma_2023} proposes a noise-aware variant of the token swapping approximation algorithm which is a subroutine in our routing method.
It covers two-qubit gate errors but is insensitive to crosstalk errors.
In future research, developing a crosstalk-aware token swapping algorithm will help the proposed routing method to further mitigate noise.
\tb{Another direction of future research is the incorporation of holistic performance metrics, such as layer fidelity or cycle benchmarking, in our routing method~\cite{Mckay2023,Erhard_2019}}

In summary, efficiently characterizing and mitigating noise is crucial to faithfully run circuits on noisy quantum devices.
Our work significantly improves noise mitigation in qubit routing compared to existing methods.
This is achieved by additionally mitigating crosstalk errors which are highly relevant for applications and in particular sampling-based applications.

\appendix
\section{Randomized Benchmarking}\label{sec:app_rb}
Randomized benchmarking is a protocol to determine average gate error rates.
In its simplest version, a random sequence of Clifford gates is applied to a set of $n$ qubits initialized in the zero-state.
A final gate is chosen such that it inverts the random sequence.
Then, a theoretical error model predicts that the probability of finding the qubits in the ground state will approximately show an exponential decay of the form
\begin{align}
    P(0) = A \cdot \alpha^m + B
\end{align}
where $m$ is the length of the random gate sequence.
The constants $A$ and $B$ absorb state preparation and readout errors as well as the error of the final gate.
The decay rate $\alpha$ relates to the average error per Clifford gate (EPC) via
\begin{align}
    EPC = \left( 1-\frac{1}{d} \right)\cdot (1-\alpha)
\end{align}
where $d=2^n$ is the dimension of the Hilbert space of $n$ qubits.

\section{Additional SRB Results\label{sec:app_xt}}
Here, we give results for additional SRB experiments.
First, to support our claim that CX-SQ crosstalk is the more relevant than SQ-CX crosstalk, we characterize SQ-CX crosstalk for the complete \emph{ibmq\_ehningen} chip via SRB, data shown in Fig.~\ref{fig:hist_srb_sqcx}.
In order to reduce the number of measurements, the influence of all neighboring qubits on a given CX is measured simultaneously by performing SRB on the CX and all of its neighbors.
This simplification can only increases the observed ERR compared to a full SQ-CX crosstalk measurement.
Analogously, we characterize SQ-SQ crosstalk for the complete chip via SRB, see Fig.~\ref{fig:hist_srb_sqsq}.
In summary, we observe error rate ratios of at most $3.1$ and $5.5$ for SQ-SQ and SQ-CX, respectively.
Thus, SQ-SQ and SQ-CX crosstalk are an order of magnitude smaller than CX-SQ.
Finally, we characterize CX-CX crosstalk via SRB.
Results are shown in Fig.~\ref{fig:hist_srb_cxcx}.
Here, we observe that if large crosstalk is measured for a particular CX-CX pair, there is also a corresponding CX-SQ measurement showing crosstalk.
From this, we conclude that CX-CX can be attributed to CX-SQ crosstalk.
\begin{figure}
    \centering
 	\subfloat{\includegraphics[width=1\textwidth]{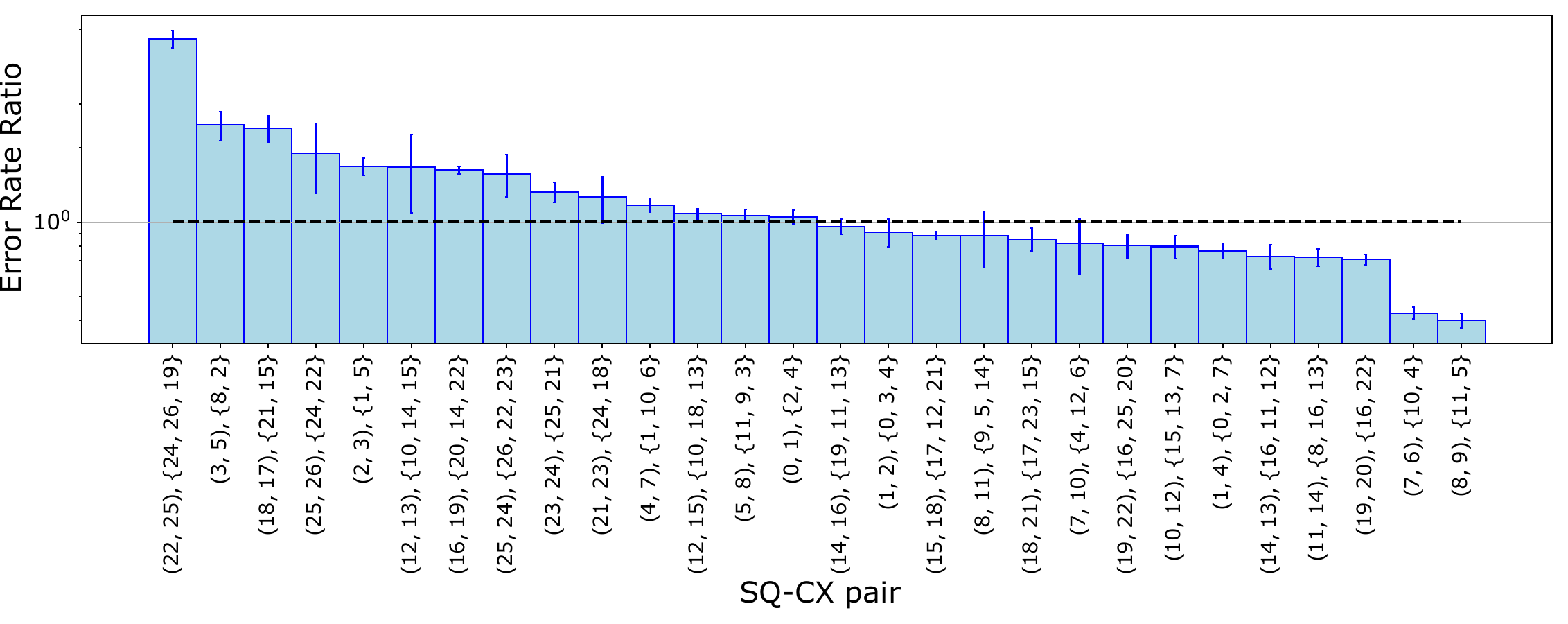}}\\
	\caption{SRB measurements of SQ-CX crosstalk.}
	\label{fig:hist_srb_sqcx}
\end{figure}
\begin{figure}
    \centering
 	\subfloat{\includegraphics[width=1\textwidth]{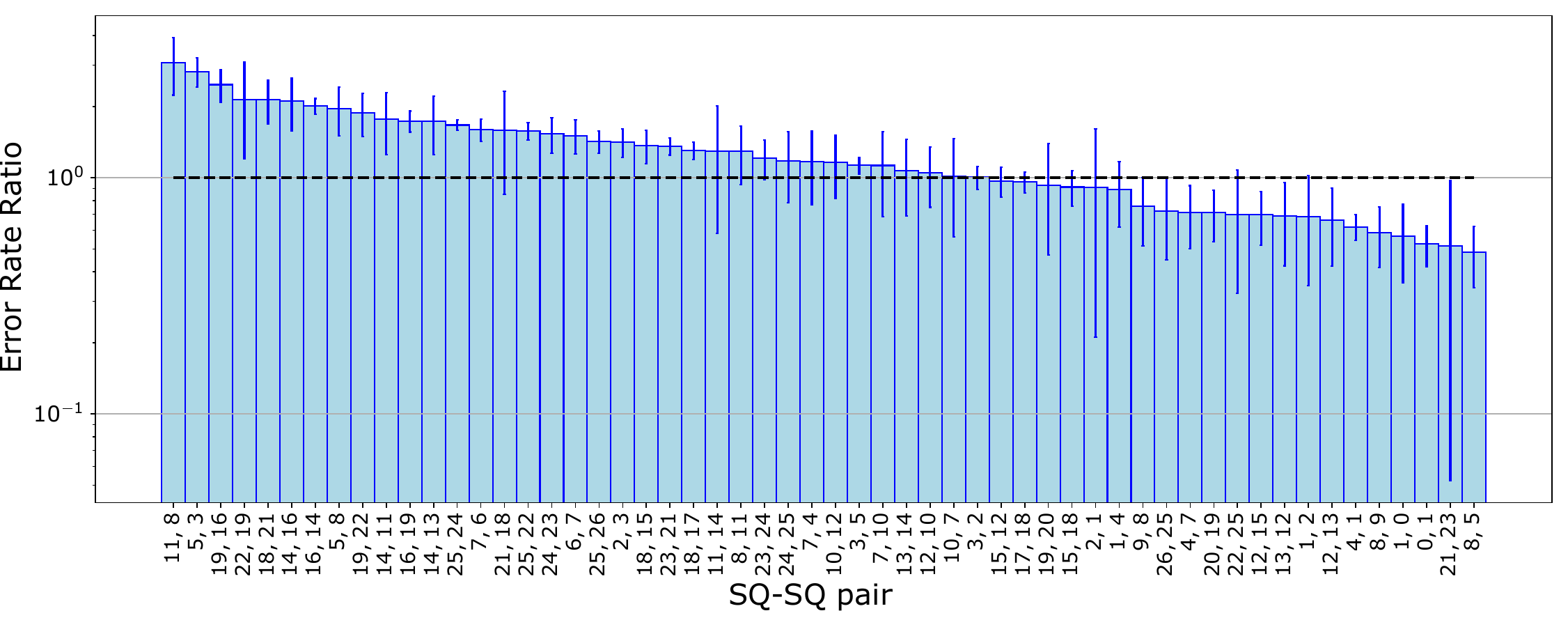}}\\
	\caption{SRB measurements of SQ-SQ crosstalk.}
	\label{fig:hist_srb_sqsq}
\end{figure}
\begin{figure}
    \centering
 	\subfloat{\includegraphics[width=1\textwidth]{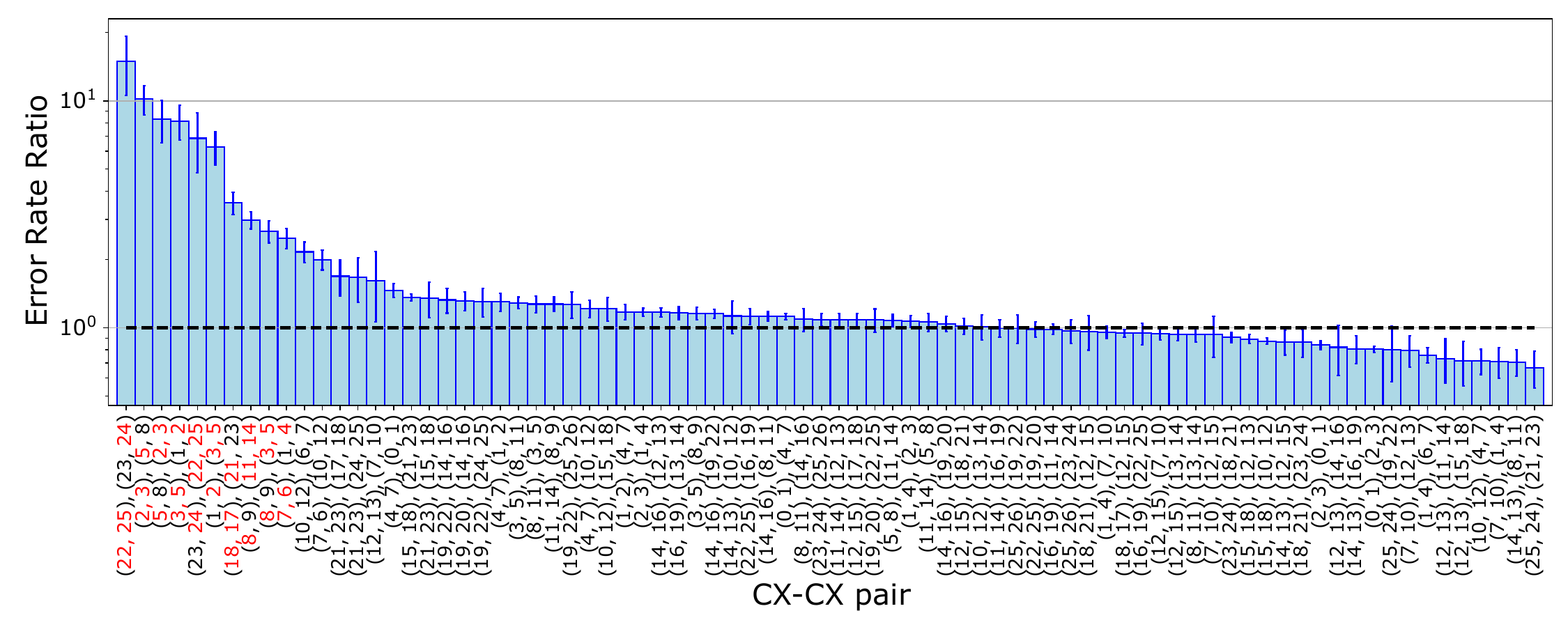}}\\
	\caption{SRB measurements of CX-CX crosstalk. For the ten largest ERRs, there is a corresponding CX-SQ pair, marked in red, showing also large ERR, compare Fig.~\ref{fig:hist_cxrb}.}
	\label{fig:hist_srb_cxcx}
\end{figure}

\newpage

\section{\tb{Additional CXRB Results}}\label{sec:app_kyoto}
Here, we show additional data for a complete characterization of CX-SQ crosstalk in the 127 qubit device \emph{ibm\_kyoto}.
Using the optimal experiment schedule from Section~\ref{sec:coloring}, we characterize all 394 CX-SQ pairs using only six consecutive batches of CXRB circuits.
The measured ERR is close to one for most CX-SQ pairs, see Fig.~\ref{fig:cxrb_kyoto}.
We observe an ERR larger than one for only 30 pairs at a statistical significance level of 95~\%.
Additionally, we perform crosstalk characterization via standard SRB.
We visualize the correlation between the ERRs measured via SRB and CXRB in Fig.~\ref{fig:cxrb_srb_cor_kyoto}.
The Pearson correlation coefficient computes to $0.32$ at a p-value of $3.0\cdot 10^{-6}$.

\begin{figure}   
    \subfloat[ERRs for all 394 CX-SQ pairs. For clarity, pair labels are not shown.]{\includegraphics[width=1\textwidth]{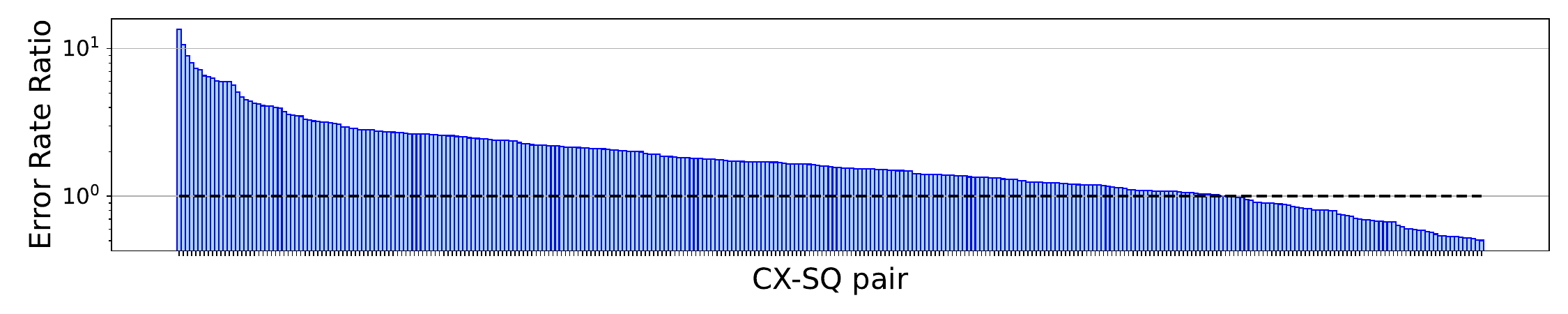}\label{fig:hist_cxrb_kyoto_full}} \\
    \subfloat[60 largest and 10 smallest ERRs.]{\includegraphics[width=1\textwidth]{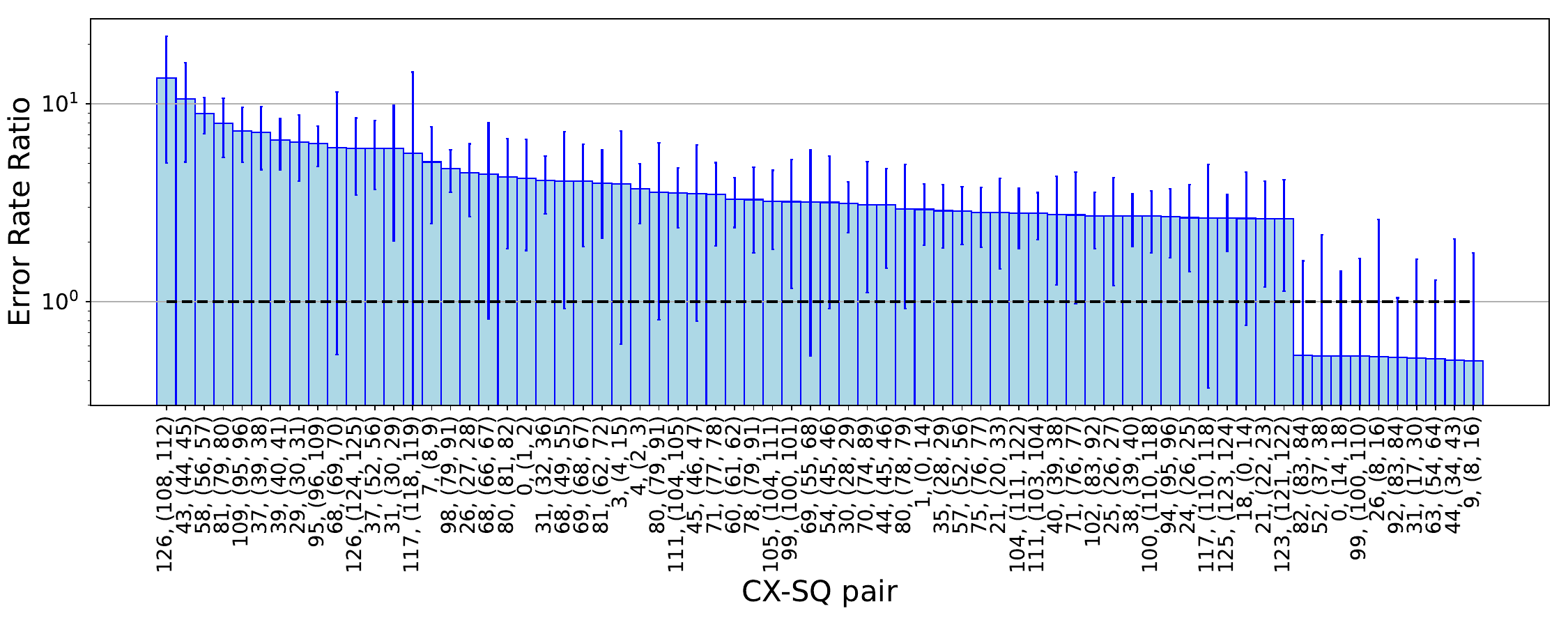}\label{fig:hist_cxrb_kyoto}}
	\caption{\tb{CX-SQ magnitude for the complete chip of \emph{ibm\_kyoto}. In (a), we visualize the ratio between the error rate with and without applied CR pulse (ERR) for all 394 CX-SQ pairs in descending order. In (b), we give the 60 largest and 10 smallest ERRs. Error bars are obtained by performing an error propagation from the errors in the exponential decay fits.
 Large error bars are likely due to the smaller sample size of 10 compared to 12 in the \emph{imbq\_ehningen} experiments of Fig.~\ref{fig:hist_cxrb}.
 The dashed line corresponds to $\text{ERR}=1$, i.e., no crosstalk.}
 }
	\label{fig:cxrb_kyoto}    
\end{figure}

\begin{figure}
	\centering
   % \subfloat{
   % \includegraphics[width=0.5\linewidth]{./fig/cor_strechedRB_SRB_kyoto}}
    \subfloat{
    \includegraphics[width=0.5\linewidth]{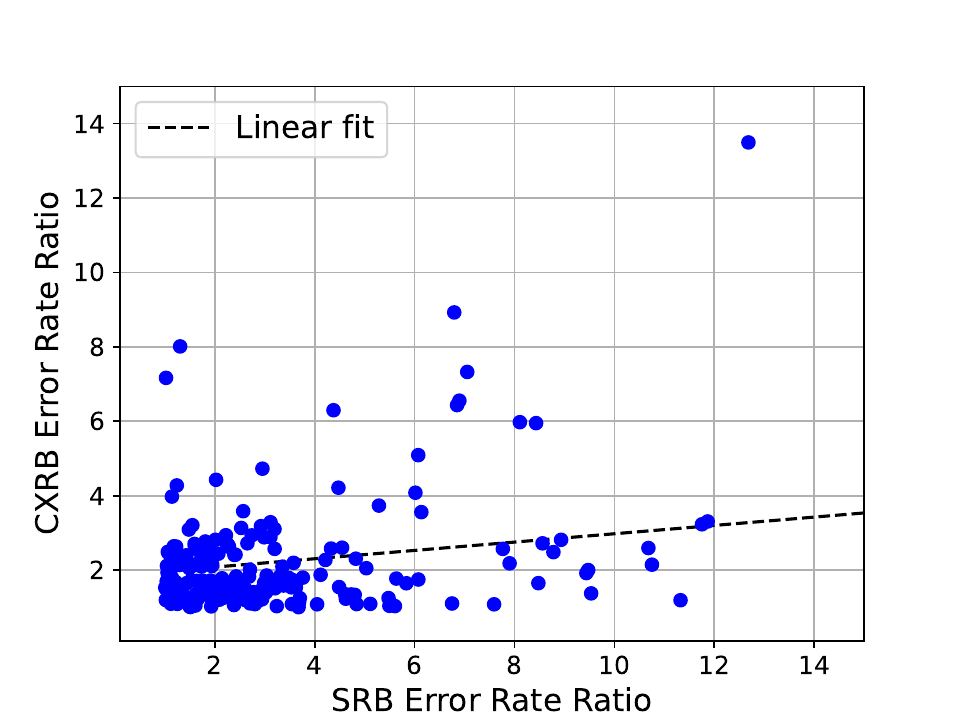}}
	\caption{
    Correlation between the error rate ratio measured via SRB and CXRB for \emph{ibm\_kyoto}.
    The Pearson correlation coefficient computes to $0.32$ at a p-value of $3.0\cdot 10^{-06}$.
    }
	\label{fig:cxrb_srb_cor_kyoto}
\end{figure}

\section{Ramsey Experiment}\label{app:ramsey}
To deepen our understanding of the CX-SQ crosstalk we perform a modified Ramsey experiment.
Ramsey experiments are intended to measure coherence time and qubit frequency.
First, a $\sqrt{X}$ pulse maps the qubit on the equator of the Bloch sphere.
After a variable delay time, another $\sqrt{X}$ pulse is applied before measuring the qubit in the computational basis.
If the true frequency of the qubit differs from the frequency of the applied frame, we will observe a time-dependent oscillation as the qubit precesses with respect to the frame.
As a result, an oscillation is observed in the qubit population.
Moreover, phase information is lost due to decoherence, leading to an exponential damping of the oscillation.
Altogether, one observes a damped oscillation in the qubit population, where the oscillation frequency equals the difference between qubit frequency and frame frequency
whereas the damping constant connects to the coherence time.

To characterize CX-SQ crosstalk between CX $(i,j)$ and qubit $k$
by Ramsey experiments,
we first conduct a standard Ramsey experiment on qubit $k$.
Afterwards, the experiment is repeated with a simultaneously applied, stretched CX pulse on qubits $(i,j)$, analogous to the CXRB experiment.
This experiment allows us to investigate whether crosstalk is caused by a shift in qubit frequency.

Exemplary results of Ramsey experiments showing such a frequency shift are shown in Fig.~\ref{fig:ramsey_ex}.
Analogously to RB, we compute the ratio between the oscillation frequency measured with a stretched CX-pulse and the oscillation frequency measured in isolation.
This ratio represents how much the difference between qubit frequency and frame frequency increases when a CX pulse is applied.
We conduct modified Ramsey experiments to characterize frequency shifts for the complete \emph{ibmq\_ehningen} chip.
Results are visualized in Fig~\ref{fig:hist_ramsey}.
We observe oscillation frequencies in the order of $10$ to $100~\mathrm{kHz}$.
The largest increase in oscillation frequency we measure is 25 when a stretched CX pulse is applied.
In Fig.~\ref{fig:ramsey_srb_cor} we compare the frequency ratio to the ERR measured via CXRB.
In general, we do not observe a significant correlation between frequency ratio and ERR.
However, for some pairs showing a large ERR we also detect a large frequency shift.
From these data we conclude that only some of the gate-induced crosstalk is due to frequency shifts of the qubits.

\begin{figure}
    \centering
 	\subfloat{\includegraphics[width=0.5\textwidth]{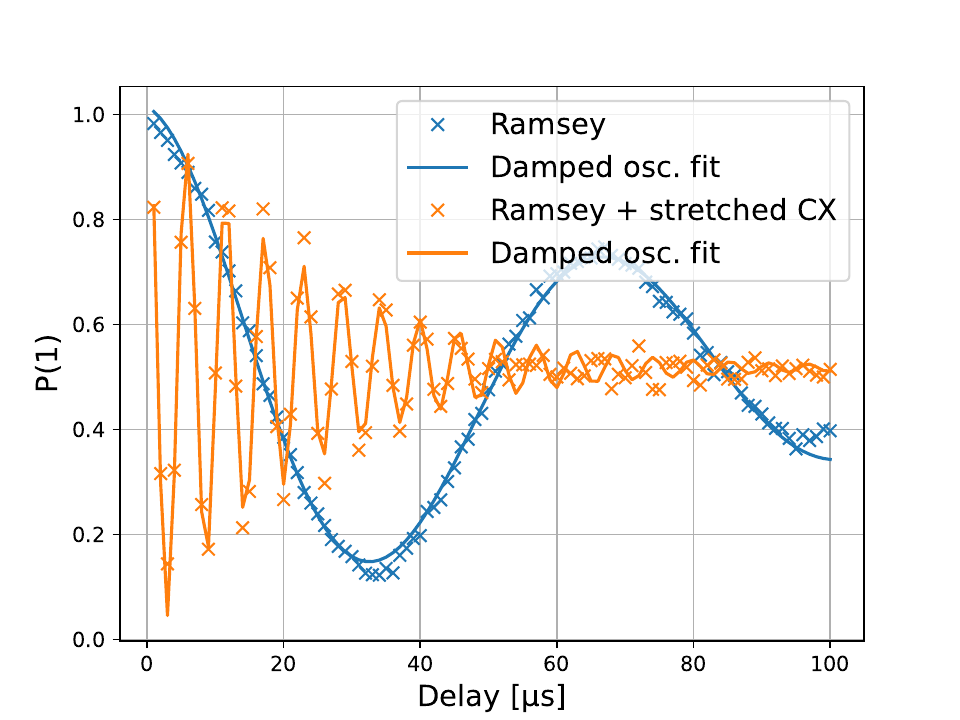}}\\
	\caption{Exemplary results for Ramsey experiment.
        The blue points and curve show the data when the Ramsey experiment is performed on qubit 24 in isolation.
        The orange points and curve show the data when a stretched cross-resonance pulse is applied on the neighboring qubit pair $(22,25)$.
        The oscillation frequency shifts from $14.6$~kHz to $176$~kHz.}
	\label{fig:ramsey_ex}
\end{figure}
\begin{figure}
    \centering
 	\subfloat{\includegraphics[width=1\textwidth]{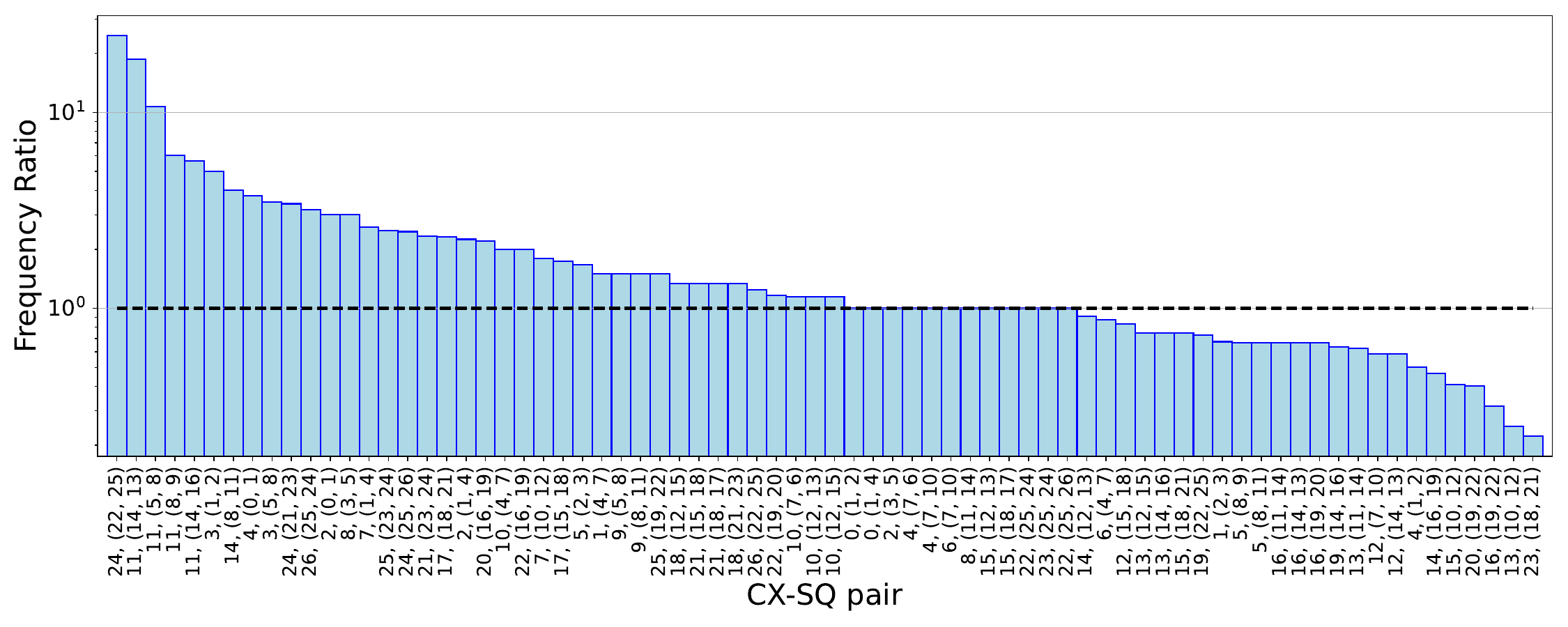}}\\
	\caption{Ramsey measurements of CX-SQ crosstalk.}
	\label{fig:hist_ramsey}
\end{figure}
\begin{figure}
	\centering
    \includegraphics[width=0.4\linewidth]{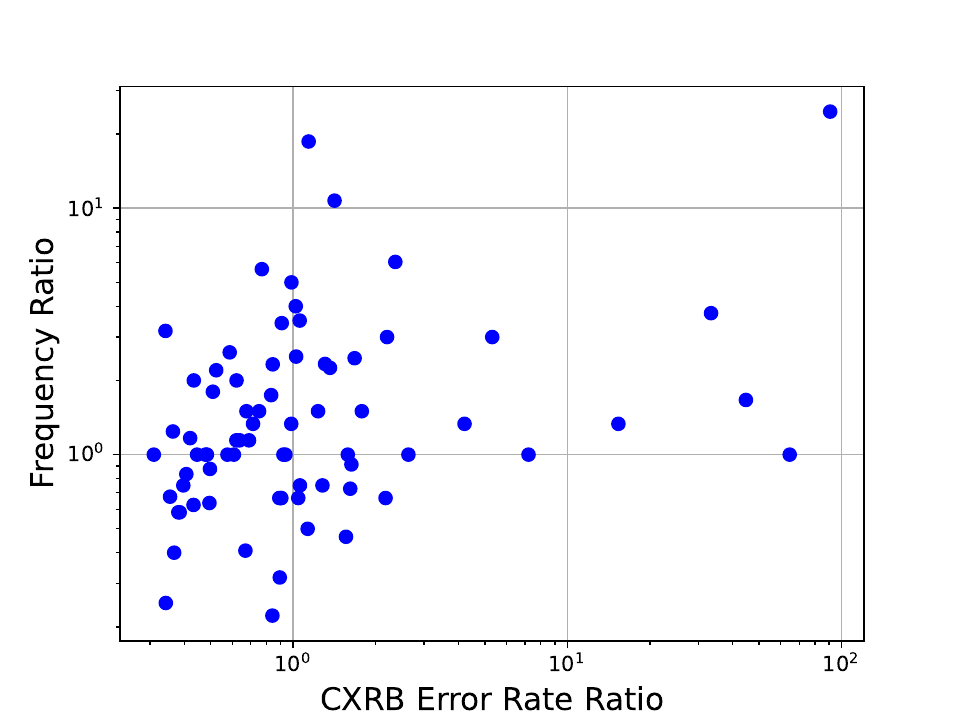}
	\caption{Correlation between the increase in error rate, measured via RB, and the frequency shift, measured via Ramsey experiments.}
	\label{fig:ramsey_srb_cor}
\end{figure}

\paragraph{Acknowledgements.}
This research is supported by the Bavarian Ministry of Economic Affairs, Regional Development and Energy with funds from the Hightech Agenda Bayern.
This research is part of the Munich Quantum Valley, which is supported by the
Bavarian state government with funds from the Hightech Agenda Bayern Plus.
	\newpage
	\bibliographystyle{unsrturl}  
	\bibliography{./bib_ct}
	
\end{document}